\pdfoutput=1
\RequirePackage{ifpdf}
\ifpdf 
\documentclass[pdftex]{sigma}
\else
\documentclass{sigma}
\fi

\usepackage{mathrsfs}
\usepackage{enumitem}
\usepackage[arrow,matrix,curve]{xy}

\numberwithin{equation}{section}

\newtheorem{Theorem}{Theorem}[section]
\newtheorem{Corollary}[Theorem]{Corollary}
\newtheorem{Lemma}[Theorem]{Lemma}
\newtheorem{Proposition}[Theorem]{Proposition}
\newtheorem{Propanddef}[Theorem]{Proposition and Definition}

\theoremstyle{definition}
\newtheorem{Definition}[Theorem]{Definition}

\newtheorem{Example}[Theorem]{Example}
\newtheorem{Remark}[Theorem]{Remark}

\newcommand{\IR}{\mathbb{R}}
\newcommand{\IN}{\mathbb{N}}
\newcommand{\IZ}{\mathbb{Z}}

\newcommand*{\longhookrightarrow}{\ensuremath{\lhook\joinrel\relbar\joinrel\rightarrow}}

\newcommand{\f}{\frac}

\newcommand{\id}{\operatorname{id}}
\newcommand{\Sym}{\mathsf{S}}
\newcommand{\garb}{\mathscr{C}^{\infty}}
\newcommand{\Verti}{\mathsf{V}}

\newcommand{\vbdl}[1]{\Verti(#1)}
\newcommand{\vbdlup}[1]{#1^\Verti}

\newcommand{\calC}{\mathscr{C}}

\newcommand{\calX}{\mathscr{X}}
\newcommand{\locC}{\mathscr{C}_{\textup{loc}}^\infty}
\newcommand{\locX}{\mathscr{X}_{\textup{loc}}^\infty}

\newcommand{\locOmega}[1]{\Omega^{#1}_{\textup{loc}}}
\newcommand{\lockOmega}{\Omega^k_{\textup{loc}}}

\newcommand{\jet}{\mathsf{j}}
\newcommand{\Jet}{\mathsf{J}}
\newcommand{\jetfib}{\mathsf{F}}

\newcommand{\IEE}{\mathsf{E}}
\newcommand{\Id}{\mathrm{d}}
\newcommand{\IT}{\mathrm{T}}

\newcommand{\beginaufz}{\begin{enumerate}[leftmargin=0cm,itemindent=.5cm,labelwidth=\itemindent,labelsep=0cm,align=left,label={\rm \alph*)}]}

\begin{document}

\allowdisplaybreaks

\newcommand{\arXivNumber}{1308.1005}

\renewcommand{\PaperNumber}{003}

\FirstPageHeading

\ShortArticleName{Formal Solution Spaces of Formally Integrable PDEs}
\ArticleName{The Prof\/inite Dimensional Manifold Structure\\ of Formal Solution Spaces\\ of Formally Integrable PDEs}

\Author{Batu G\"UNEYSU~$^\dag$ and Markus J.~PFLAUM~$^\ddag$}

\AuthorNameForHeading{B.~G\"uneysu and M.J.~Pf\/laum}

\Address{$^\dag$~Institut f\"ur Mathematik, Humboldt-Universit\"at, Rudower Chaussee 25, 12489 Berlin, Germany}
\EmailD{\href{mailto:gueneysu@math.hu-berlin.de}{gueneysu@math.hu-berlin.de}}

\Address{$^\ddag$~Department of Mathematics, University of Colorado, Boulder CO 80309, USA}
\EmailD{\href{mailto:markus.pflaum@colorado.edu}{markus.pflaum@colorado.edu}}
\URLaddressD{\url{http://math.colorado.edu/~pflaum/}}

\ArticleDates{Received March 30, 2016, in f\/inal form January 05, 2017; Published online January 10, 2017}

\Abstract{In this paper, we study the formal solution space of a nonlinear PDE in a~f\/iber bundle. To this end, we start with foundational material and introduce the notion of a pfd structure to build up a new concept of prof\/inite dimensional manifolds. We show that the inf\/inite jet space of the f\/iber bundle is a prof\/inite dimensional manifold in a~natural way. The formal solution space of the nonlinear PDE then is a subspace of this jet space, and inherits from it the structure of a prof\/inite dimensional manifold, if the PDE is formally integrable. We apply our concept to scalar PDEs and prove a~new criterion for formal integrability of such PDEs. In particular, this result entails that the Euler--Lagrange equation of a~relativistic scalar f\/ield with a polynomial self-interaction is formally integrable.}


\Classification{58A05; 58A20; 35A30}

\section{Introduction}
Even though it appears to be unsolvable in general, the problem to describe the moduli space of solutions of a~particular nonlinear PDE has led to powerful new results in geometric analysis and mathematical physics. Notably this can be seen, for example, by the fundamental work on the structure of the moduli space of
Yang--Mills equations \cite{AtiHitSin, Don, Tau}. Among the many challenging problems which arise when studying moduli spaces of solutions of nonlinear PDEs
is that the space under consideration does in general not have a manifold structure, usually not even one modelled on an inf\/inite dimensional Hilbert or Banach space. Moreover, the solution space can possess singularities. A way out of this dilemma is to study compactif\/ications of the moduli space like the completion of the moduli space with respect to a certain Sobolev metric, cf.~\cite{FreUhl}. Another way, and that is the one we are advocating in this article, is to consider a~``coarse'' moduli space consisting of so-called formal solutions of a PDE, i.e., the space of those smooth functions whose power series expansion at each point solves the PDE. In case the PDE is formally integrable in a sense def\/ined in this article, the formal solution space turns out to be a prof\/inite dimensional manifold. These possibly inf\/inite dimensional spaces are ringed spaces which can be regarded as projective limits of projective systems of f\/inite dimensional manifolds.

Prof\/inite dimensional manifolds appear naturally in several areas of mathematics, in par\-ticular in deformation quantization, see for example \cite{PflDQSO}, the structure theory of Lie-projective groups~\cite{BicLPG,hof}, in connection with functional integration on spaces of connections~\cite{ash}, and in the secondary calculus invented by Vinogradov \cite{krasvino,KraLycVinGJSNPDE, vino} which inspired the approach in this paper, cf.~in particular \cite[Chapter~7]{KraLycVinGJSNPDE}. It is to be expected that the theory of prof\/inite dimensional manifolds as set up in this paper will have further applications, for example for the (classical) perturbation theory of PDEs in mathematical physics, where one should understand a perturbation as a deformation, i.e., a smooth family of prof\/inite dimensional solution manifolds of a~(formally integrable) PDE depending on a parameter. Work on this is in progress.

The paper consists of two main parts. The f\/irst, Section~\ref{pfd}, lays out the foundations of the theory of prof\/inite dimensional manifolds. Besides the papers~\cite{bern} and~\cite{AbbManDSPLM}, where the latter is taylored towards explaining the dif\/ferential calculus by Ashtekar and Lewandowski~\cite{ash}, literature on prof\/inite dimensional manifolds is scarce. Moreover, our approach to prof\/inite dimensional manifolds is novel in the sense that we def\/ine them as ringed spaces together with a so-called pfd structure, which consists not only of one but a whole equivalence class of representations by projective systems of f\/inite dimensional manifolds. The major point hereby is that all the projective systems appearing in the pfd structure induce the same structure sheaf, which allows to def\/ine dif\/ferential geometric concepts depending only on the pfd structure and not a particular representative. One way to construct dif\/ferential geometric objects is by dualizing projective limits of manifolds to injective limits of, for example, dif\/ferential forms, and then sheaf\/ify the thus obtained presheaves of ``local'' objects. Again, it is crucial to observe that these sheaves are independent of the particular choice of a~representative within the pfd structure, whereas the ``local'' objects obtain a f\/iltration which depends on the choice of a particular representative.
Using variants of this approach or directly the structure sheaf of smooth functions, we introduce in Section~\ref{pfd} tangent bundles of prof\/inite dimensional manifolds and their higher tensor powers, vector f\/ields, and dif\/ferential forms.

The second main part is Section~\ref{fsnlpde}, where we introduce the formal solution space of a nonlinear PDE. We f\/irst explain the necessary concepts from jet bundle theory and on prolongations of PDEs in f\/iber bundles, following essentially Goldschmidt \cite{gold}, cf.~also \cite{krasvino, pomm,vino,vita}. In Section~\ref{lini} we introduce in the jet bundle setting a notion of an operator symbol of a nonlinear PDE such that, in the linear case, it coincides with the well-known (principal) symbol of a partial dif\/ferential operator up to canonical isomorphisms. The corresponding result, Proposition~\ref{PropLinSymb}, appears to be folklore; see~\cite{spencer69} and~\cite[Section~IV.2]{KraLycVinGJSNPDE} for related work. Afterwards, we show that the bundle of inf\/inite jets is a prof\/inite dimensional manifold. This result immediately entails that the formal solution space of a formally integrable PDE is a prof\/inite dimensional submanifold of the inf\/inite jet bundle. Finally, in Section~\ref{scal}, we consider scalar PDEs. We prove there a widely applicable criterion for the formal integrability of scalar PDEs, which to our knowledge has not appeared in the mathematical literature yet. Moreover, we conclude from our criterion that the Euler--Lagrange equation of a relativistic scalar f\/ield with a polynomial self-interaction on an arbitrary Lorentzian manifold is formally integrable, so its formal solution space is a prof\/inite dimensional manifold. We expect that this observation will be of avail when clarifying the Poisson structure~\cite{marsden,vita,zuckerman} and quantization theory~-- possibly through deformation~-- of such scalar f\/ield theories, cf.~\cite{duetschfredenhagen,rejzner}.

\section{Some notation}

Let us introduce some notation and conventions which will be used throughout the paper.

If nothing else is said, all manifolds and corresponding concepts, such as submersions, bundles etc., are understood to be smooth and f\/inite dimensional.
The symbol $\IT^{k,l}$ stands for the functor of $k$-times contravariant and $l$-times covariant tensors, where as usual $\IT:=\IT^{1,0}$ and $\IT^{*}:=\IT^{0,1}$. If $X$ is a manifold, then the corresponding tensor bundles will be denoted by~$\pi_{\IT^{k,l} X}\colon \IT^{k,l} X\to X$. Moreover, we write $ \calX^\infty$ and $ \Omega^k$ for the sheaves of smooth vector f\/ields and of smooth $k$-forms, respectively.

Given a f\/ibered manifold, i.e., a surjective submersion $\pi\colon E\to X$, we write $\Gamma^{\infty}(\pi)$ for the sheaf of smooth sections of $\pi$. Its space of sections over an open $U \subset X$ will be denoted by~$\Gamma^{\infty}(U;\pi)$. The set of \emph{local smooth sections of} $\pi$ \emph{around a point} $p\in M$ is the set of smooth sections def\/ined on some open neighborhood of $p$ and will be denoted by $\Gamma^{\infty}(p;\pi)$. The stalk at $p$ then is a quotient space of $\Gamma^{\infty}(p;\pi)$ and is written as $\Gamma^{\infty}_p(\pi)$.

The \emph{vertical vector bundle} corresponding to the f\/ibered manifold $\pi$ is def\/ined as the subvector bundle
\begin{gather*}
 \vbdlup{\pi}\colon \ \vbdl{\pi}:=\operatorname{ker}(\IT\pi) \longrightarrow E
\end{gather*}
of $\pi_{\IT E}\colon \IT E\rightarrow E$. If $\pi'\colon E'\to X$ is a second f\/ibered manifold, the \emph{vertical morphism} corresponding to a morphism $h\colon E \to E'$ of f\/ibered manifolds over $X$ is given by
\begin{gather*}
 \vbdlup{h} \colon \ \vbdl{\pi}\longrightarrow \vbdl{\pi'} ,\qquad v \longmapsto \IT h (v).
\end{gather*}

If $\pi\colon E\to X$ is a vector bundle, then the f\/ibers of $\pi$ are $\IR$-vector spaces, hence one can apply tensor functors f\/iberwise to obtain the corresponding tensor bundles. In particular, $\pi^{\odot^k}\colon \Sym^k(\pi)\to X$ will stand for the \emph{$k$-fold symmetric tensor product bundle of} $\pi$.

Finally, unless otherwise stated, the notions ``projective system'' and ``projective limit'' will always be understood in the category of topological spaces,
where they of course exist; see \cite[Chapter~VIII, Section~3]{EilSteFAT}. In fact, given such a projective system $( M_i , \mu_{ij})_{i,j\in\IN, i\leq j}$, a distinguished projective limit is given as follows. Def\/ine
\begin{gather*}
 M:= \bigg\{ (p_i)_{i\in \IN}\in \prod_{i\in \IN} M_i\,|\, \mu_{ij} (p_j) = p_i \text{ for all $i,j\in \IN$ with $i \leq j$} \bigg\}
\end{gather*}
to be the subspace of all threads in the product, and the continuous maps $\mu_j \colon M \rightarrow M_j$ as the restrictions of the canonical projections
$\prod_{i\in \IN} M_i \rightarrow M_j$ to $M$. Then one obviously has $\mu_{ij} \circ \mu_j =\mu_i$ for all $i,j \in \IN$ with $i\leq j$. Note that a basis of the topology of $M$ is given by the set of all open sets of the form $\mu_i^{-1} (U)$, where $i\in \IN$ and $U \subset M_i$ is open. In the following, we will refer to the thus def\/ined $M$ together with the maps $(\mu_i)_{i\in\IN}$ as \emph{the canonical projective limit} of $( M_i , \mu_{ij})_{i,j\in\IN, i\leq j} $, and denote it by $M= \lim\limits_{\longleftarrow \atop i\in \IN} M_i$.

\section{Prof\/inite dimensional manifolds}\label{pfd}

In this section, we introduce the concept of prof\/inite dimensional manifolds and establish the dif\/ferential geometric foundations of this new category.
For comparison and further reading on this topic we refer to \cite{AbbManDSPLM,DodGalVasGFC} and \cite[Section~1.4]{PflDQSO}.

\subsection{The category of prof\/inite dimensional manifolds}

The following def\/inition lies in the center of the paper:

\begin{Definition}\quad
 \begin{enumerate}\itemsep=0pt
 \item[a)] By a \emph{smooth projective system} we understand a family $( M_i , \mu_{ij} )_{i,j \in \IN, i \leq j}$ of smooth mani\-folds $M_i$ and surjective submersions $\mu_{ij} \colon M_j \rightarrow M_i $ for $i \leq j$ such that the following conditions hold true:
 \begin{enumerate}[itemindent=0mm,leftmargin=3.5em,labelwidth=0mm,labelsep=2mm,align=right,label={\rm (SPS\arabic*)}]\itemsep=0pt
 \item\label{IteSPS1}
 $\mu_{ii} =\id_{M_i}$ for all $i \in \IN$.
 \item\label{IteSPS2}
 $\mu_{ij} \circ \mu_{jk} = \mu_{ik}$ for all $i,j,k \in \IN$ such that $ i \leq j \leq k$.
 \end{enumerate}

 \item[b)] If $( M_a' , \mu_{ab}' )_{a,b \in \IN, \, a \leq b}$ denotes a second smooth projective system, a \emph{morphism of smooth projective systems} between $( M_i , \mu_{ij} )_{i,j \in \IN,\, i \leq j}$ and $( M_a' , \mu_{ab}' )_{a,b \in \IN,\, a \leq b}$ is a pair $( \varphi , (F_a)_{a \in \IN})$ consisting of a strictly increasing map $\varphi \colon \IN \rightarrow \IN$ and a family of smooth maps $F_a\colon$ $M_{\varphi (a)} \rightarrow M_a'$, $a \in \IN$ such that for each pair $a,b \in \IN$ with $a \leq b$ the diagram
\begin{gather*}
 \xymatrix{
 M_{\varphi (a)}\ar[d]_{F_a} &&\ar[ll]_{\mu_{\varphi(a) \varphi(b)}} M_{ \varphi (b)} \ar[d]^{F_b} \\
 M_a' && \ar[ll]_{\mu_{ab}'} M_b' }
\end{gather*}
 commutes. We usually denote a smooth projective system shortly by
 $\big( M_i , \mu_{ij} \big)$ and write
 \begin{gather*}
( \varphi , F_a)\colon \ ( M_i , \mu_{ij} ) \longrightarrow ( M_a' , \mu_{ab}')
 \end{gather*}
to indicate that $( \varphi , (F_a)_{a\in \IN})$ is a morphism of smooth projective systems. If each of the maps $F_a$ is a submersion (resp.~immersion), we call the morphism $( \varphi , F_a)$ a \emph{submersion} (resp.~\emph{immersion}).

 \item[c)]
 Two smooth projective systems $( M_i , \mu_{ij})$ and $( M_a', \mu_{ab}')$ are called \emph{equivalent}, if there are surjective submersions
 \begin{gather*}
( \varphi , F_a )\colon \ ( M_i , \mu_{ij} ) \longrightarrow ( M_a', \mu_{ab}' ),\qquad ( \psi , G_i)\colon \ ( M_a', \mu_{ab}' )
 \longrightarrow ( M_i, \mu_{ij} )
 \end{gather*}
 such that the diagrams
 \begin{gather*} 
 \xymatrix{ M_i && \ar[ll]_{\mu_{i\,\varphi(\psi(i))} } \ar[dl]^{F_{\psi (i)}} M_{\varphi(\psi(i))} \\
 & M_{\psi (i)}' \ar[ul]^{G_i}
 } \qquad \text{and}\qquad
 \xymatrix{ M_a' && \ar[ll]_{\mu_{a\,\psi( \varphi(a))}' } \ar[dl]^{G_{\varphi(a)}} M_{\psi(\varphi(a))}' \\
 & M_{\varphi (a)} \ar[ul]^{F_a}
 }
 \end{gather*}
 commute for all $i,a\in \IN$. A pair of such surjective submersions will be called an \emph{equivalence transformation of smooth projective systems}.
\end{enumerate}
\end{Definition}

\begin{Remark}
In the def\/inition of smooth projective systems and later in the one of smooth projective representations we use the partially ordered set $\IN$ as index set. Obviously, $\IN$ can be replaced there by any partially ordered set canonically isomorphic to $\IN$ such as an inf\/inite subset of $\IZ$ bounded from below. We will silently use this observation in later applications for convenience of notation.
\end{Remark}

\begin{Example}\label{Ex:sps} \quad
\begin{enumerate}\itemsep=0pt
\item[a)]
 Let $M$ be a manifold. Then $( M_i, \mu_{ij})$ with $M_i: =M$ and $\mu_{ij} := \id_M$ for $i\leq j$ is a smooth projective system which we call \emph{constant} and which we denote shortly by $(M, \id_M)$.

\item[b)] Assume that for $i\leq j$ one has given surjective linear maps $\lambda_{ij} \colon V_j \rightarrow V_i$ between real f\/inite dimensional vector spaces such that~\ref{IteSPS1} and~\ref{IteSPS2} are satisf\/ied. Then $( V_i , \lambda_{ij}) $ is a~smooth projective system. For example, this
 situation arises in deformation quantization of symplectic manifolds when constructing the completed symmetric tensor algebra of a~f\/inite dimensional real vector space; see~\cite{PflDQSO} for details. Of course, a simpler example is given by the canonical projections $\pi_{ij} \colon \IR^j \rightarrow \IR^i$ onto the f\/irst $i$ coordinates, hence $( \IR^i , \pi_{ij}) $ is a~(non-trivial) smooth projective system.

 \item[c)] 
 In the structure theory of topological groups \cite{BicLPG, hof} one considers smooth projective systems $(\mathsf{G}_i,\eta_{ij})$ such that each $\mathsf{G}_j$ is a Lie group and the $\eta_{ij}\colon \mathsf{G}_j\to \mathsf{G}_i$ are continuous group homomorphisms. See Example~\ref{Ex:pfdmfd}(c) below for a precise description of the projective limits of such projective systems of Lie groups.

 \item[d)] The tower of $k$-jets over a f\/iber bundle together with their canonical projections forms a~smooth projective system (see Section~\ref{finjet}).
\end{enumerate}
\end{Example}

Within the category of (smooth f\/inite dimensional) manifolds, a projective limit of a smooth projective system obviously does in general not exist. In the following, we will enlarge the category of manifolds by the so-called prof\/inite dimensional manifolds (and appropriate morphisms). The thus obtained category will contain projective limits of smooth projective systems.

\begin{Definition}\quad
\begin{enumerate}\itemsep=0pt
\item[a)] By a \emph{smooth projective representation} of a commutative locally $\IR$-ringed space $(M,\calC^\infty_M)$ we understand a smooth projective system $( M_i , \mu_{ij})$ together with a family of continuous maps $\mu_i \colon M \rightarrow M_i$, $i\in\IN$, such that the following conditions hold true:
\begin{enumerate}[itemindent=0mm,leftmargin=4em,labelwidth=0mm,labelsep=2mm,align=right,label={\rm (PFM\arabic*)}]\itemsep=0pt
\item\label{It:ProjLim}
 As a topological space, $M$ together with the family of maps $\mu_i$, $i\in\IN$, is a~projective limit of $( M_i , \mu_{ij})$.
\item\label{Ite:StrucSheaf}
 The section space $\calC^\infty_M(U)$ of the structure sheaf over an open subset $U\subset M$ is given by the set of all $f \in \calC (U)$ such that for every $p\in U$ there exists an $i\in \IN$, an open $U_i \subset M_i$ and an $f_i\in \calC^\infty(U_i)$ such that $p\in \mu_i^{-1} (U_i) \subset U$ and
 \begin{gather*}
 f_{|\mu_i^{-1} (U_i)} = f_i \circ {\mu_i}_{|\mu_i^{-1} (U_i)}
 \end{gather*}
 hold true.
\end{enumerate}
We usually denote a smooth projective representation brief\/ly as a family $( M_i , \mu_{ij} , \mu_i )$.
\item[b)] A smooth projective representation $( M_i , \mu_{ij}, \mu_i)$ of $(M,\calC^\infty_M)$ is said to be \emph{regular}, if each of the maps $\mu_{ij} \colon M_j \rightarrow M_i$ is a f\/iber bundle.
\item[c)] Two smooth projective representations $( M_i , \mu_{ij} , \mu_i )$ and $( M_a', \mu_{ab}' , \mu_a')$ of $(M,\calC^\infty_M)$ are called
 \emph{equivalent}, if there is an equivalence transformation of smooth projective systems
 \begin{gather*}
( \varphi , F_a )\colon \ ( M_i , \mu_{ij} ) \longrightarrow ( M_a', \mu_{ab}' ),\qquad ( \psi , G_i)\colon \ ( M_a', \mu_{ab}' )
 \longrightarrow ( M_i , \mu_{ij} )
 \end{gather*}
 such that
 \begin{gather*}
 \mu_i = G_i \circ \mu_{\psi (i)}' \qquad\text{and} \qquad \mu_a' = F_a \circ \mu_{\varphi (a)}
 \qquad\text{for all} \quad i,a\in \IN.
 \end{gather*}
In the following, we will sometimes call such a pair of surjective submersions an \emph{equivalence transformation of smooth projective representations}. The equivalence class of a smooth projective system $( M_i , \mu_{ij} ,\mu_i)$ will be simply denoted by $[( M_i , \mu_{ij} ,\mu_i)]$ and called a~\emph{pfd structure} on $(M,\calC^\infty_M)$.
\end{enumerate}
\end{Definition}

\begin{Proposition}\label{Prop:equivrepsystems}
Let $(M,\calC^\infty_M)$ be a commutative locally $\IR$-ringed space with a smooth projective representation $( M_i , \mu_{ij} , \mu_i )$. Assume further that $ ( M_a', \mu_{ab}')$ is a smooth projective system which is equivalent to $( M_i , \mu_{ij})$. Then there are continuous maps $\mu_a' \colon M \rightarrow M_a'$, $a\in\IN$, such that $( M_a', \mu_{ab}' , \mu_a' )$ becomes a smooth projective representation of $(M,\calC^\infty_M)$ which is equivalent to $( M_i , \mu_{ij} , \mu_i)$.
\end{Proposition}

\begin{proof}
 Choose an equivalence transformation of smooth projective systems
 \begin{gather*}
 ( \varphi , F_a )\colon \ ( M_i , \mu_{ij} ) \longrightarrow ( M_a', \mu_{ab}' ),\qquad
 ( \psi , G_i)\colon \ ( M_a', \mu_{ab}' ) \longrightarrow ( M_i , \mu_{ij} ).
 \end{gather*}
Put $\mu_a' := F_a \circ \mu_{\varphi (a)}$. Let us show f\/irst that $M$ together with the family of continuous maps~$\mu_a'$, $a \in \IN$ is a projective limit of $\big( M_a', \mu_{ab}' \big)$. So assume that $X$ is a topological space, and $h_a\colon X \rightarrow M_a'$, $a\in \IN$ a family of continuous maps such that $h_a = \mu_{ab}' \circ h_b$ for $a\leq b$. Since $M$ is a projective limit of $( M_i , \mu_{ij})$, there exists a uniquely determined $h\colon X \rightarrow M$ such that $\mu_i \circ h = G_i \circ h_{\psi (i)}$ for all $i\in \IN$. But then
 \begin{gather*}
 \mu_a' \circ h = F_a\circ \mu_{\varphi (a)} \circ h = F_a \circ G_{\varphi (a)}
 \circ h_{\psi (\varphi (a))} = \mu_{a\psi(\varphi(a))}' \circ h_{\psi (\varphi (a))} = h_a .
 \end{gather*}
 Moreover, if $\widetilde h\colon X \rightarrow M$ is a continuous function such that
 $ \mu_a' \circ \widetilde h = h_a$ for all $a\in \IN$, one computes
 \begin{gather*}
 \mu_i \circ \widetilde h = \mu_{i\varphi (\psi (i))} \circ \mu_{\varphi (\psi (i))} \circ \widetilde h =
 G_i \circ F_{\psi(i)} \circ \mu_{\varphi (\psi (i))} \circ \widetilde h = G_i \circ \mu_{\psi (i)}' \circ
 \widetilde h = G_i \circ h_{\psi (i)}.
 \end{gather*}
 Since $M$ is a projective limit of $ ( M_i , \mu_{ij})$, this entails $\widetilde h =h$. This proves that $M$ is a projective limit of $( M_a', \mu_{ab}')$.

Next let us show that~\ref{Ite:StrucSheaf} holds true with the $\mu_i$ replaced by the $\mu_a'$. So let $U\subset M$ be open, $f \in \calC^\infty_M(M)$, and $p\in U$. Choose $i\in \IN$ such that there is an open $U_i \subset M_i$ and a smooth $f_i\colon U_i \rightarrow \IR$ with $p\in \mu_i^{-1} (U_i) \subset U$ and $f_{|\mu_i^{-1} (U_i)} = f_i \circ {\mu_i}_{|\mu_i^{-1} (U_i)}$. Put $a:=\psi (i)$, $V_a := G_i^{-1} (U_i)$, and def\/ine $\widetilde f_a \colon V_a \rightarrow \IR$ by $\widetilde f_a := f_i \circ {G_i}_{|V_a}$. Then $\widetilde f_a$ is smooth, and
 \begin{align*}
 \widetilde f_a \circ {\mu_a'}_{|{\mu_a'}^{-1} (V_a)} & = f_i \circ G_i \circ
 F_{\psi(i)} \circ {\mu_{\varphi(\psi(i))}}_{|{\mu_a'}^{-1} (V_a)}\\
 & = f_i \circ \mu_{i\varphi(\psi(i))} \circ {\mu_{\varphi(\psi(i))}}_{|{\mu_a'}^{-1}
 (V_a)} = f_i \circ {\mu_i}_{|{\mu_a'}^{-1} (V_a)} = f_{|{\mu_a'}^{-1} (V_a)},
 \end{align*}
 where we have used that ${\mu_a'}^{-1} (V_a) = \mu_i^{-1} (U_i)$. Similarly one shows that a continuous $\widetilde f\colon U \rightarrow \IR$ is an element of $\calC^\infty_M(U)$, if for every $p\in U$ there is an $a\in \IN$, an open $V_a\subset M_a'$, and a smooth function $\widetilde f_a \colon V_a \rightarrow \IR$ such that $p\in {\mu_a'}^{-1} (V_a) \subset U$ and $\widetilde f_a \circ {\mu_a'}_{|{\mu_a'}^{-1} (V_a)} = \widetilde f_{|{\mu_a'}^{-1} (V_a)}$.

 Finally, it remains to prove that $\mu_i = G_i \circ \mu_{\psi (i)}'$ for all $i\in \IN$, but this follows from
 \begin{gather*}
 G_i \circ \mu_{\psi (i)}' = G_i \circ F_{\psi(i)} \circ \mu_{\varphi (\psi(i)))} = \mu_{i \varphi (\psi(i)))} \circ \mu_{\varphi (\psi(i)))} = \mu_i .
 \end{gather*}
 This f\/inishes the proof.
\end{proof}

\begin{Remark} The preceding proposition entails that the structure sheaf of a commutative locally $\IR$-ringed space $( M,\calC^\infty_M)$ for which a smooth projective representation $( M_i , \mu_{ij} , \mu_i)$ exists depends only on the equivalence class $[( M_i , \mu_{ij} , \mu_i )]$.
\end{Remark}

The latter remark justif\/ies the following def\/inition:

\begin{Definition}\quad
\begin{enumerate}\itemsep=0pt
\item[a)] By a \emph{profinite dimensional manifold} we understand a commutative locally $\IR$-ringed space $(M,\calC^\infty_M)$ together with a pfd structure def\/ined on it. The prof\/inite dimensional manifold $(M,\calC^\infty_M)$ is called \emph{regular}, if there exists a~regular smooth representation within the pfd structure on $(M,\calC^\infty_M)$.
\item[b)]
 Assume that $(M,\calC^\infty_M)$ and $(N,\calC^\infty_N)$ are prof\/inite dimensional manifolds. Then a continuous map
 $f\colon M \rightarrow N$ is said to be \emph{smooth}, if the following condition holds true:
 \begin{itemize}\itemsep=0pt
 \item[]
 for every open $U\subset N$, and $g \in \calC^\infty_N(U)$ one has
 \begin{gather*}
 g \circ f_{|f^{-1} (U)} \in \calC^\infty_M\big( f^{-1} (U) \big) .
 \end{gather*}
 \end{itemize}
\end{enumerate}
\end{Definition}

By def\/inition, it is clear that the composition of smooth maps between prof\/inite dimensional manifolds is smooth, hence prof\/inite dimensional manifolds and the smooth maps between them as morphisms form a category, the isomorphisms of which can be safely called \emph{diffeomorphisms}. All of this terminology is justif\/ied by the simple observation Example~\ref{Ex:pfdmfd}(a) below.

\begin{Example}\label{Ex:pfdmfd}\quad
\begin{enumerate}\itemsep=0pt
\item[a)] 
Given a manifold $M$, the constant smooth projective system $( M,\id_M)$ def\/ines a smooth projective representation for the ringed space $(M,\calC^\infty_M)$. Hence, every manifold is a~pro\-f\/i\-ni\-te dimensional manifold in a natural way, and the category of manifolds a full subcategory of the category of prof\/inite dimensional manifolds.

\item[b)] 
Assume that $( M_i , \mu_{ij})$ is a smooth projective system. Let
\begin{gather*}
 M:= \lim\limits_{\longleftarrow \atop i \in \IN} M_i
\end{gather*}
together with the natural projections $\mu_i \colon M\rightarrow M_i$ denote the canonical projective limit of $\big( M_i , \mu_{ij} \big)$. Then, \ref{It:ProjLim} is fulf\/illed by assumption, and it is immediate that $M$ carries a~uniquely determined structure sheaf $\calC^\infty_M$ which satisf\/ies \ref{Ite:StrucSheaf}. The locally ringed space $(M,\calC^\infty_M)$ together with the pfd structure $[( M_i , \mu_{ij} , \mu_i)]$ then is a prof\/inite dimensional manifold. This prof\/inite dimensional manifold is even a projective limit of the projective system $( M_i , \mu_{ij})$ within the category of prof\/inite dimensional manifolds. We therefore write in this situation
 \begin{gather*}
 \big(M, \calC^\infty_M\big) =\lim \limits_{\longleftarrow \atop i\in \IN } \big(M_i,\calC^\infty_{M_i}\big)
\end{gather*}
and call $(M, \calC^\infty_M)$ (together with $[( M_i , \mu_{ij} , \mu_i)]$) \emph{the canonical smooth projective limit} of $( M_i , \mu_{ij})$.

\item[c)]
A locally compact Hausdorf\/f topological group $\mathsf{G}$ is called \emph{Lie projective}, if every neighbourhood of the identity contains a compact Lie normal subgroup, i.e., a normal subgroup $N \subset G$ such that $G/N$ is a~Lie group. One has the following structure theorem \cite[Theorem~4.4]{BicLPG}, \cite{hof}. A~locally compact metrizable group $\mathsf{G}$ is Lie projective, if and only if there is a smooth projective system $(\mathsf{G}_i,\eta_{ij})$ as in Example~\ref{Ex:sps}(c) together with continuous group homomorphisms $\eta_i\colon \mathsf{G}\to \mathsf{G}_i$, $i\in\IN$ such that $(\mathsf{G},\eta_i)$ is a projective limit of $(\mathsf{G}_i,\eta_{ij})$. Again, it follows that $\mathsf{G}$ carries a uniquely determined structure sheaf $\calC^\infty_{\mathsf{G}}$ satisfying \ref{Ite:StrucSheaf}. The locally ringed space $(\mathsf{G},\calC^\infty_{\mathsf{G}})$ together with the pfd structure $[( \mathsf{G}_i , \eta_{ij} , \eta_i)]$ becomes a~re\-gular prof\/inite dimensional manifold with a group structure such that all of its structure maps are smooth.

\item[d)] The space of inf\/inite jets over a f\/iber bundle canonically is a prof\/inite dimensional manifold (see Section~\ref{inf}).
\end{enumerate}
\end{Example}

\begin{Remark}\quad
\begin{enumerate}\itemsep=0pt
\item[a)] In the sequel, $(M,\calC^\infty_M)$ or brief\/ly $M$ will always denote a prof\/inite dimensional manifold. Moreover, $( M_i , \mu_{ij} , \mu_i)$ always stands for a smooth projective representation def\/ining the pfd structure on $M$. The sheaf of smooth functions on a~prof\/inite dimensional manifold will often brief\/ly be denoted by~$\calC^\infty$, if no confusion can arise.
\item[b)] One of the intentions when constructing the category of prof\/inite dimensional manifolds was that it should be a category with projective limits which extends the one of smooth manifolds and that it is minimal in a certain sense with respect to these properties. The category of prof\/inite dimensional manifolds fulf\/ills these requirements. That it extends the category of manifolds follows from Example~\ref{Ex:pfdmfd}(a). By a~straightforward argument using the `diagonal trick' for doubly projective limits one concludes that the category of prof\/inite dimensional manifolds contains all projective limits. The minimality requirement is a direct consequence of the def\/inition of prof\/inite dimensional manifolds as abstract projective limits of manifolds.
\item[c)] The prof\/inite dimensional manifolds def\/ined in this paper coincide with the projective limits of manifolds from~\cite{AbbManDSPLM}, but are in general not plb-manifolds in the sense of \cite[Def\/inition~3.1.2]{DodGalVasGFC}. The latter have the property that they can be modelled locally on Fr\'echet spaces representable as projective limits of Banach spaces. A prof\/inite dimensional manifold of inf\/inite dimension, though, can in general not locally be modelled by open subsets of $\IR^\infty$. In particular when the underlying prof\/inite dimensional manifold is given as the manifold of formal solutions of a formally integrable PDE in the sense of Proposition and Def\/inition \ref{defS} corresponding local charts with values $\IR^\infty$ appear to exist only in particular cases. A more detailed study of this phenomenon is left for future work.
\end{enumerate}
\end{Remark}

Let $N\subset M$ be a subset, and assume further that for some smooth projective representation $( M_i , \mu_{ij} , \mu_i)$ of the pfd structure on $M$ the following holds true:
 \begin{enumerate}[itemindent=0mm,leftmargin=4.5em,labelwidth=0mm,labelsep=2mm,align=right,label={\rm (PFSM\arabic*)}]
 \item \label{axiomsubmfd}
 There is a stricly increasing sequence $(l_i)_{i \in \IN}$ such that for every $i\in \IN$ the set $N_i:= \mu_{l_i}(N)$ is a submanifold of $M_{l_i}$.
 \item \label{axiomcap}
 One has $N=\bigcap\limits_{i\in \IN} \mu_{l_i}^{-1}(N_i)$.
 \item \label{axiomsubmers}
 The induced map
 \begin{gather*}
 \nu_{ij}:= {\mu_{l_i l_j}}_{|N_j} \colon \ N_j \longrightarrow N_i
 \end{gather*}
 is a submersion for all $i,j\in\IN$ with $j\geq i$.
 \end{enumerate}
Observe that the $\nu_{ij}$ are surjective by def\/inition of the manifolds $N_i$ and by $\nu_i = \nu_{ij} \circ \nu_j$, where we have put $\nu_i := {\mu_{l_i}'}_{|N}$. In particular, $(N_i,\nu_{ij})$ becomes a smooth projective system.

\begin{Propanddef}Let $N \subset M$ be a subset such that for some smooth projective representation $(M_i,\mu_{ij},\mu_i)$ of the pfd structure on $M$ the axioms~\ref{axiomsubmfd} to~\ref{axiomsubmers} are fulfilled. Then $N$ carries in a natural way the structure of a profinite dimensional manifold such that its sheaf of smooth functions coincides with the sheaf $\calC^\infty_{|N}$ of continuous functions on open subset of $N$ which are locally restrictions of smooth functions on $M$. A smooth projective representation of $N$ defining its natural pfd structure is given by the family $(N_i,\nu_{ij},\nu_i)$. From now on, such a subset $N \subset M$ will be called a~\emph{prof\/inite dimensional submanifold of}~$M$, and $(M_i,\mu_{ij},\mu_i)$ a~smooth projective representation of $M$ \emph{inducing the submanifold structure on}~$N$.
\end{Propanddef}

\begin{proof} We f\/irst show that $N$ together with the maps $\nu_i$ is a (topological) projective limit of the projective system $(N_i,\nu_{ij})$. Let $p_i \in N_i$, $i\in\IN$ such that $\nu_{ij} (p_j) =p_i$ for all $j\geq i$. Since $M$ together with the $\mu_{i}$ is a projective limit of $(M_i,\mu_{ij})$, there exists an $p\in M$ such that $\mu_{l_i} (p) =p_i$ for all $i \in\IN$. By axiom \ref{axiomcap}, $p \in N$, hence one concludes that $N$ is a projective limit of the manifolds $N_i$.

Next, we show that $\calC^\infty_{|N}$ coincides with the uniquely determined sheaf $\calC^\infty_N$ satisfying axiom \ref{Ite:StrucSheaf}. Since the canonical embeddings $N_i \hookrightarrow M_{l_i}$ are smooth by~\ref{axiomsubmfd}, the embedding $N \hookrightarrow M$ is smooth as well, and $\calC^\infty_{|N}$ is a subsheaf of the sheaf $\calC^\infty_N$. It remains to prove that for every open $V\subset N$ a function $f\in \calC^\infty_N (V)$ is locally the restriction of a smooth function on~$M$. To show this let $p\in V$ and $V_i$ an open subset of some $N_i$ such that $p\in \nu_i^{-1} (V_i) \subset V$, and such that there is an $f_i \in \calC^\infty (V_i)$ with $f_{|\nu_i^{-1} (V_i)} = f_i \circ {\nu_i}_{| \nu_i^{-1} (V_i)}$. Since $N_i$ is locally closed in~$M_{l_i}$, we can assume after possibly shrinking $V_i$ that there is an open $U_i \subset M_{l_i}$ with $V_i = N_i \cap U_i$ and such that $N_i \cap U_i$ is closed in $U_i$. Then there exists ${F_i}\in \calC^\infty (U_i)$ such that ${F_i}_{|V_i} = f_i$. Put $F := F_i \circ {\mu_{l_i}}_{|\mu_{l_i}^{-1} (U_i)}$. Then $F \in \garb(\mu_{l_i}^{-1} (U_i))$, and
\begin{gather*}
 f_{|\nu_i^{-1} (V_i)} = F_{|\nu_i^{-1} (V_i)},
\end{gather*}
which proves that $f \in \calC^\infty_{|N} (V)$. The claim follows.
\end{proof}

\begin{Example}\quad
\begin{enumerate}\itemsep=0pt
\item[a)] Every open subset $U$ of $M$ is naturally a prof\/inite dimensional submanifold since for each $i\in\IN$ the set $U_i := \mu_i (U)$ is an open submanifold of $M_i$.
\item[b)] Consider the prof\/inite dimensional manifold
 \begin{gather*}
 \big( \IR^\infty , \calC^\infty_{\IR^\infty}\big) := \lim\limits_{\longleftarrow \atop n \in \IN} \big( \IR^n ,\calC^\infty_{\IR^n}\big),
 \end{gather*}
 and let $\mathrm{B}^n (0)$ be the open unit ball in $\IR^n$. The projective limit
 \begin{gather*}
 \big( \mathrm{B}^\infty (0) , \calC^\infty_{\mathrm{B}^\infty (0)}\big) :=
 \lim \limits_{\longleftarrow \atop n \in \IN} \big( \mathrm{B}^n (0), \calC^\infty_{\mathrm{B}^n (0)} \big)
 \end{gather*}
 then becomes a prof\/inite dimensional submanifold of $\IR^\infty$. Note that it is not locally closed in~$\IR^\infty$.
\item[c)] The space of formal solutions of a formally integrable partial dif\/ferential equation is a~pro\-f\/inite dimensional submanifold of the space of inf\/inite jets over the underlying f\/iber bundle (see Section~\ref{inf}).
\end{enumerate}
\end{Example}

We continue with:

\begin{Definition} Let $U \subset M$ be open. A smooth function $f \in \calC^\infty (U)$ then is called \emph{local}, if there is an open $U_i \subset M_i$ for some $i\in \IN$ and a function $f_i \in \calC^\infty (U_i)$ such that $U \subset \mu_i^{-1} (U_i)$ and $f = f_i \circ {\mu_i}_{|U}$. We denote the space of local functions over~$U$ by~$\locC(U)$.
\end{Definition}

\begin{Remark}\quad
\begin{enumerate}\itemsep=0pt
\item[a)] Observe that $\locC$ forms a presheaf on $M$, which depends only on the pfd structure $[( M_i , \mu_{ij} , \mu_i)]$. Moreover, it is clear by construction that for every open $U\subset M$ and every representative $(M_i,\mu_{ij},\mu_i)$ of the pfd structure, $\locC (U)$ together with the family of pull-back maps $\mu_i^* \colon \calC^\infty (\mu_i (U)) \rightarrow \locC (U)$ is an inductive limit of the injective system of linear spaces $( \calC^\infty (\mu_i (U)) , \mu_{ij}^*)_{i\in \IN}$.
\item[b)] $\locC$ is in general not a sheaf unless $M$ is a f\/inite dimensional manifold. The sheaf associated to $\locC$ naturally coincides with $\garb$ since locally, every smooth function is local.
\item[c)] By naming sections of $\locC$ local functions we essentially follow Stashef\/f \cite[Def\/inition~1.1]{StaSHABVA} and Barnich \cite[Def\/inition~1.1]{glenn1}, where the authors consider jet bundles. Note that in \cite{AbbManDSPLM}, local functions are called cylindrical functions.
\item[d)] The representative $\mathcal{M}:=(M_i,\mu_{ij},\mu_j)$ leads to a particular f\/iltration $\mathcal{F}^{\mathcal{M}}_\bullet$ of the presheaf of local functions by putting, for $l\in \IN$,
\begin{gather*}
 \mathcal{F}^{\mathcal{M}}_l\big(\locC \big):= \mu_l^* \calC^\infty_{M_l} .
\end{gather*}
Observe that this f\/iltration has the property that
\begin{gather*}
 \locC =\bigcup_{l\in\IN}\mathcal{F}^{\mathcal{M}}_l\big(\locC\big) .
\end{gather*}
\end{enumerate}
\end{Remark}

\subsection{Tangent bundles and vector f\/ields}

The tangent space at a point of a f\/inite dimensional manifold can be def\/ined as a~set of equivalence classes of germs of smooth paths at that point or as the space of derivations on the stalk of the sheaf of smooth functions at that point. The def\/inition via paths can not be immediately carried over to the prof\/inite dimensional case, so we use the derivation approach.

\begin{Definition}
 Given a point $p$ of the prof\/inite dimensional manifold $M$, the \emph{tangent space} of~$M$ at~$p$ is def\/ined as the space of derivations on $\calC^\infty_p$, the stalk of smooth functions at $p$, i.e., as the space
 \begin{gather*}
 \IT_p M := \operatorname{Der} \big( \calC^\infty_p , \IR \big) .
 \end{gather*}
 Elements of $\IT_pM$ will be called \emph{tangent vectors} of $M$ at $p$. The \emph{tangent bundle} of $M$ is the disjoint union
 \begin{gather*}
 \IT M:=\bigcup_{p\in M} \IT_p M ,
 \end{gather*}
 and
 \begin{gather*}
 \pi_{\IT M}\colon \ \IT M \longrightarrow M, \qquad \IT_pM \ni Y \longmapsto p
 \end{gather*}
 the \emph{canonical projection}.
\end{Definition}

Note that for every $i\in \IN$ there is a canonical map $\IT \mu_i \colon \IT M \rightarrow \IT M_i$ which maps a tangent vector $Y \in \IT_pM$ to the tangent vector
\begin{gather*}
 Y_i \colon \ \calC^\infty_{M_i, p_i} \rightarrow \IR, \qquad [f_i]_{p_i} \mapsto Y\big( [f_i \circ \mu_i]_p \big) , \qquad \text{where} \quad p_i := \mu_i(p).
\end{gather*}
By construction, one has $\IT \mu_{ij} \circ\IT \mu_j = \IT \mu_i $ for $i\leq j$. We give $\IT M$ the coarsest topology such that all the maps $\IT \mu_i$, $i\in \IN$ are continuous. Now we record the following observation:

\begin{Lemma}\label{tangi}
The topological space $\IT M$ together with the maps $\IT \mu_i$ is a projective limit of the projective system $( \IT M_i, \IT \mu_{ij})$.
\end{Lemma}

\begin{proof}
Assume that $X$ is a topological space, and $\big( \Phi_i \big)_{i\in \IN}$ a family of continuous maps $\Phi_i \colon X \rightarrow \IT M_i$ such that $\IT \mu_{ij} \circ \Phi_j = \Phi_i $ for all $i\leq j$. Since $M$ is a projective limit of the projective system $\big( M_i , \mu_{ij} \big)$, there exists a~uniquely determined continuous map $ \varphi \colon X \rightarrow M$ such that $ \pi_{\IT M_i} \circ \Phi_i = \mu_i \circ \varphi$ for all $i\in \IN$. Now let $x\in X$, and put $p:= \varphi (x)$ and $p_i := \mu_i(p)$. Then, for every $i\in \IN$, $\Phi_i (x)$ is a tangent vector of $M_i$ with footpoint $p_i$.
We now construct a derivation $\Phi(x) \in \operatorname{Der}( \calC^\infty_p , \IR)$. Let $[f]_p \in \calC^\infty_p$, i.e., let $f$ be a smooth function def\/ined on a neighborhood~$U$ of~$p$, and $[f]_p$ its germ at~$p$. Then there exists $i\in \IN$, an open neighborhood $U_i \subset M_i$ of~$p_i$ and a smooth function $f_i \colon U_i \rightarrow \IR$ such that
\begin{gather*}
 \mu_i^{-1} (U_i) \subset U \qquad \text{and} \qquad f_{|\mu_i^{-1} (U_i)} = f_i \circ {\mu_i}_{|\mu_i^{-1} (U_i)}.
\end{gather*}
We now put
\begin{gather*} 
 \Phi(x)\big( [f]_p \big) := \Phi_i (x) \big( [f_i]_{p_i} \big), \qquad \text{where} \quad p_i := \mu_i (p).
\end{gather*}
We have to show that $\Phi(x)$ is independent of the choices made, and that it is a derivation indeed. So let $f' \colon U' \rightarrow \IR$ be another smooth function def\/ining the germ $[f]_p$. Choose $j\in \IN$, an open neighborhood $U_j' \subset M_j$ of $p_j$, and a smooth function $f_j' \colon U_j' \rightarrow \IR$ such that
\begin{gather*}
 \mu_j^{-1} (U_j') \subset U' \qquad \text{and} \qquad f_{|\mu_j^{-1} (U_j')} = f_j' \circ {\mu_j}_{|\mu_j^{-1} (U_j')}.
\end{gather*}
Without loss of generality, we can assume $i\leq j$. By assumption $[f]_p = [f']_p$, hence one concludes that
\begin{gather*}
 f_i\circ {\mu_{ij}}_{|V_j} = {f_j'}_{|V_j}
\end{gather*}
for some open neighborhood $V_j \subset M_j$ of $p_j := \mu_j (p)$. But this implies, using the assumption on the $\Phi_i$ that
\begin{gather*}
 \Phi_j (x) \big( [f_j']_{p_j} \big) = \IT \mu_{ij} \Phi_j (x) \big( [f_i]_{p_i} \big) = \Phi_i (x) \big( [f_i]_{p_i} \big) .
\end{gather*}
Hence, $ \Phi(x)$ is well-def\/ined, indeed.

Next, we show that $ \Phi(x)$ is a derivation. So let $[f]_p, [g]_p \in \calC^\infty_p$ be two germs of smooth functions at $p$. Then, after possibly shrinking the domains of~$f$ and~$g$, one can f\/ind an $i\in \IN$, an open neighborhood $U_i\subset M_i$ of $p_i$, and $f_i,g_i \in \calC^\infty (U_i)$ such that
\begin{gather*}
 f_{|\mu_i^{-1} (U_i)} = f_i \circ{\mu_i}_{|\mu_i^{-1} (U_i)} \qquad \text{and} \qquad g_{|\mu_i^{-1} (U_i)} = g_i \circ{\mu_i}_{|\mu_i^{-1} (U_i)} .
\end{gather*}
Since $\Phi_i (x)$ acts as a derviation on $\calC^\infty_{p_i}$, one checks
\begin{align*}
 \Phi(x) \big( [f]_p [g]_p \big) & = \Phi_i(x) \big( [f_i]_{p_i} [g_i]_{p_i} \big) = f_i(p_i) \Phi_i(x) \big( [g_i]_{p_i} \big) + g_i(p_i) \Phi_i(x) \big( [f_i]_{p_i} \big) \\
 & = f(p) \Phi(x) \big( [g]_p \big)+ g(p) \Phi(x) \big( [f]_p \big),
\end{align*}
which means that $\Phi(x)$ is a derivation.

By construction, it is clear that
\begin{gather*}
 \IT\mu_i \Phi (x) = \Phi_i (x) \qquad \text{for all} \quad i\in \IN.
\end{gather*}
Let us verify that $\Phi(x)$ is uniquely determined by this property. So assume that $\Phi'(x)$ is another element of $\IT_pM$ such that $\IT\mu_i \Phi' (x) = \Phi_i (x)$ for all $i\in \IN$. For $[f]_p \in \calC^\infty_p$ of the form $f = f_i \circ {\mu_i}_{|\mu_i^{-1} (U_i)}$ with $U_i\subset M_i$ an open neighborhood of $p_i$ and $f_i \in \calC^\infty (U_i)$ this assumption entails
\begin{gather*}
 \Phi (x) \big( [f]_p \big) = \Phi_i (x) \big( [f_i]_{p_i} \big) = \Phi' (x) \big( [f]_p \big) .
\end{gather*}
Since every germ $[f]_p$ is locally of the form $f_i \circ {\mu_i}_{|\mu_i^{-1} (U_i)}$, we obtain $\Phi(x)=\Phi'(x)$.

Finally, we observe that $\Phi \colon X \rightarrow \IT M $ is continuous, since all maps $\Phi_i = \IT \mu_i \Phi$ are continuous, and $\IT M$ carries the initial topology with respect to the maps $\IT \mu_i$.

This concludes the proof that $\IT M$ together with the maps $\IT \mu_i$ is a projective limit of the projective system $( \IT M_i , \IT \mu_{ij})$.
\end{proof}

\begin{Remark}\quad{\samepage
\begin{enumerate}\itemsep=0pt
\item[a)] If $p\in M$, $Y_p, Z_p \in \IT_p M$, and $\lambda \in \IR$, then the maps $Y_p+Z_p \colon \calC^\infty_p \rightarrow \IR$ and $\lambda Y_p \colon \calC^\infty_p \rightarrow \IR$ are derivations again. Hence $\IT_pM$ becomes a topological vector space in a natural way and one has $\IT_pM\cong \lim\limits_{\longleftarrow \atop i \in \IN} \IT_{\mu_i(p)}M_i$ canonically as topological vector spaces. In particular, this implies that $\pi_{\IT M} \colon \IT M \rightarrow M$ is a continuous family of vector spaces. Note that this family need not be locally trivial, in general.

\item[b)] Denote by $\mathscr{P}^\infty_{M,p}$ the set of germs of smooth paths $\gamma \colon (\IR,0) \rightarrow (M,p)$. There is a canonical map $\mathscr{P}^\infty_{M,p} \rightarrow \IT_pM$ which associates to each germ of a smooth path $\gamma \colon (\IR,0) \rightarrow (M,p)$ the derivation
 \begin{gather*}
 \dot{\gamma}\colon \ \calC^\infty_p \longrightarrow \IR, \qquad [f]_p \longmapsto (f\circ \gamma){\dot{\mbox{\hspace{1mm}}}\,}(0) .
 \end{gather*}
 Unlike in the f\/inite dimensional case, this map need not be surjective, in general, as Example~\ref{Ex:DerNoPath} below shows. But note the following result.
\end{enumerate}}
\end{Remark}

\begin{Proposition}
 In case the profinite dimensional manifold $M$ is regular, the ``dot map''
 \begin{gather*}
 \mathscr{P}^\infty_{M,p} \longrightarrow \IT_pM, \qquad [\gamma]_0 \longmapsto \dot{\gamma} (0)
 \end{gather*}
 is surjective for every $p\in M$.
\end{Proposition}
\begin{proof} We start with an auxiliary construction. Choose a smooth projective representation $(M_i,\mu_{ij},\mu_i)$ within the pfd structure on~$M$ such that all $\mu_{ij}$ are f\/iber bundles. Put $p_i:= \mu_i (p)$ for every $i\in \IN$. Then choose a relatively compact open neighborhood $U_0\subset M_0$ of $p_0$ which is dif\/feoemorphic to an open ball in some $\IR^n$. In particular, $U_0$ is contractible, hence the f\/iber bundle ${\mu_{01}}_{|\mu_{01}^{-1} (U_0)} \colon \mu_{01}^{-1} (U_0) \rightarrow U_0$ is trivial with typical f\/iber $F_1 := \mu_{01}^{-1} (p_0)$. Let $\Psi_0 \colon \mu_{01}^{-1} (U_0) \rightarrow U_0 \times F_1$ be a trivialization of that f\/iber bundle, and $D_1\subset F_1$ an open neighborhood of $p_1$ which is dif\/feomeorphic to an open ball in some Euclidean space. Put $U_1 := \Psi_0^{-1} (U_0 \times D_1)$. Then, $U_1$ is dif\/feomeorphic to a ball in some Euclidean space, and ${\mu_{01}}_{|U_1} \colon U_1 \rightarrow U_0$ is a trivial f\/iber bundle with f\/iber $D_1$. Assume now that we have constructed $U_0\subset M_0, \ldots , U_j\subset M_j$ such that for all $i\leq j$ the following holds true:
 \begin{enumerate}\itemsep=0pt
 \item[1)] the set $U_i$ is a relatively compact open neighborhood of $p_i$ dif\/feomorphic to an open ball in some Euclidean space,
 \item[2)] for $i>0$, the identity $\mu_{i-1i} (U_i) = U_{i-1}$ holds true,
 \item[3)] for $i > 0$, the restricted map ${\mu_{i-1i}}_{|U_i} \colon U_i \rightarrow U_{i-1}$ is a trivial f\/iber bundle with f\/iber $D_i$ dif\/feomorphic to an open ball in some Euclidean space.
 \end{enumerate}
 Let us now construct $U_{j+1}$ and $D_{j+1}$. To this end note f\/irst that ${\mu_{jj+1}}_{|\mu_{jj+1}^{-1} (U_j)} \colon \mu_{jj+1}^{-1} (U_j) \rightarrow U_j$ is a trivial f\/iber bundle with typical f\/iber $F_j := \mu_{jj+1}^{-1} (p_j)$, since $U_j$ is contractible. Choose a~trivialization $\Psi_{j+1} \colon {\mu_{jj+1}}_{|\mu_{jj+1}^{-1} (U_j)} \rightarrow U_j \times F_j$, and an open neighborhood $D_{j+1} \subset F_{j+1}$ of $p_{j+1}$ which is dif\/feomorphic to an open ball in some Euclidean space. Put $U_{j+1}:= \Psi_{j+1}^{-1} \big( U_j \times D_{j+1} \big)$. Then, $U_j$ is dif\/feomeorphic to a ball in some
 Euclidean space, and ${\mu_{jj+1}}_{|U_{j+1}} \colon U_{j+1} \rightarrow U_j$ is a~trivial f\/iber bundle with f\/iber $D_{j+1}$. This f\/inishes the induction step, and we obtain $U_i\subset M_i$ and $D_i$ such that the three conditions above are satisf\/ied.

 After these preliminaries, assume that $Z\in T_pM$ is a tangent vector. Let $Z_i := \IT\mu_i (Z)$ for $i\in \IN$. We now inductively construct smooth paths $\gamma_i \colon \IR \rightarrow U_i$ such that
 \begin{gather} \label{Eq:PathProp}
 \gamma_i(0) = p_i, \qquad \dot{\gamma}_i(0) = Z_i, \qquad \text{and, if $i>0$,} \qquad \mu_{i-1i} \circ \gamma_i = \gamma_{i-1}.
 \end{gather}
To start, choose a smooth path $\gamma_0 \colon \IR \rightarrow U_0$ such that $\gamma_0(0) = p_0$, and $\dot{\gamma}_0(0) = Z_0$. Assume that we have constructed $\gamma_0,\ldots,\gamma_j$ such that \eqref{Eq:PathProp} is satisf\/ied for all $i \leq j$. Consider the trivial f\/iber bundle ${\mu_{jj+1}}_{|U_{j+1}} \colon U_{j+1} \rightarrow U_j$, and let $\Psi_{j+1}\colon U_{j+1} \rightarrow U_j \times D_{j+1}$ be a trivialization. Then, $\IT \Psi_{j+1} (Z_{j+1}) = \big( Z_j , Y_{j+1}\big)$ for some tangent vector $Y_{j+1}\in \IT_{p_{j+1}}D_{j+1}$. Choose a smooth path $\varrho_{j+1} \colon \IR \rightarrow D_{j+1}$ such that $ \varrho_{j+1}(0) = p_{j+1}$, and $\dot{\varrho}_{j+1} (0) = Y_{j+1}$. Put
 \begin{gather*}
 \gamma_{j+1} (t) = \Psi_{j+1}^{-1} \big( \gamma_j(t), \varrho_{j+1} (t)\big) \qquad \text{for all} \quad t\in \IR.
 \end{gather*}
By construction, $\gamma_{j+1}$ is a smooth path in $U_{j+1}$ such that~\eqref{Eq:PathProp} is fulf\/illed for $i=j+1$. This f\/inishes the induction step, and we obtain a family of smooth paths $\gamma_i$ with the desired properties.

Since $M$ is the smooth projective limit of the $M_i$, there exists a uniquely determined smooth path $\gamma\colon \IR \rightarrow M$ such that $\mu_i \circ \gamma = \gamma_i$ for all $i\in \IN$. In particular, this entails $\gamma (0) =p$, and $\dot{\gamma} (0) = Z$, or in other words that $Z$ is in the image of the map $ \mathscr{P}^\infty_{M,p} \rightarrow \IT_pM$.
\end{proof}

\begin{Example}\label{Ex:DerNoPath}This example shows that there exist prof\/inite dimensional manifolds having tangent vectors which can not be represented as the derivative of a~smooth path. Denote by $rS^k \subset \IR^{k+1}$ for $k\in \IN^*$ the $k$-sphere of radius $r > 0$. Moreover, denote for $1 \leq i < j$ by $\mu_{ij} \colon \IR^{j+1} \to \IR^{i+1}$ the projection onto the f\/irst $i+1$ coordinates. We use the same symbols for restrictions of $\mu_{ji}$ to open subsets. Now we def\/ine inductively a pfd system $( M_i,\mu_{ij})$ with $M_i \subset \IR^{i+1}$ open as follows:
\begin{gather*}
 M_0 := \IR, \!\quad M_1 := \IR^2 \setminus S^1, \!\quad
 M_2 := ( M_1 \times \IR) \setminus \tfrac 12 S^2, \!\quad \ldots , \!\quad M_{i+1} :=( M_i \times \IR ) \setminus \tfrac{1}{i+1} S^{i+1}.
\end{gather*}
Observe that all $\mu_{ij}$ are still surjective submersions when regarded as mappings from~$M_j$ to $M_i$. Next consider the point $p = (p_i)_{i\in \IN} \in M:=\lim\limits_{\longleftarrow \atop i\in \IN}M_i$, where $p_i := 0 \in M_i$. Now let $Y \in T_pM$ be the tangent vector represented by the family $(Y_i)_{i\in \IN^*}$ of tangent vectors
\begin{gather*}
 Y_i \colon \ \calC^\infty_{M_i, 0} \rightarrow \IR , \qquad [f]_0 \mapsto \frac{\partial f}{\partial x_1} (0) ,
\end{gather*}
where $(x_1,\ldots , x_i)$ are the canonical coordinates of $\IR^i$. Assume that there is a~smooth path $\gamma \colon ( - \varepsilon , \varepsilon ) \to M$ such that $\gamma (0) = p$ and $\dot{\gamma} (0) =Y$. Let $\gamma_i := \mu_i \circ \gamma$. Since $\dot{\gamma}_1 (0) = 1$, one can achieve after possibly shrinking $\varepsilon$ that $\dot{\gamma}_1 (t) > \frac 12$ for all $t \in ( - \varepsilon , \varepsilon )$. This implies by the mean value theorem that
$|\gamma (t)| \geq \frac 12 |t|$ for all $t \in ( - \varepsilon , \varepsilon )$. Now choose $i \in \IN^*$ such that $\frac 1i < \frac 14 \varepsilon$. Then
$\gamma_i (0)=0$ but $\gamma_i( \frac12 \varepsilon )$ has to be outside the connected component of $0$ in $M_i$. This is a~contradiction, so there does not exist a path $\gamma$ with the claimed properties, and $Y$ is not induced by a smooth path.
\end{Example}

Let us def\/ine a structure sheaf $\calC^\infty_{\IT M}$ on $\IT M$. To this end call a continuous map $f \in \calC (U)$ def\/ined on an open set $U\subset \IT M$ \emph{smooth}, if for every tangent vector $Z \in U$ there is an $i\in \IN$, an open neighborhood $U_i \subset \IT M_i$ of $Z_i:= \IT \mu_i (Z) $, and a smooth map $f_i \in \calC^\infty (U_i)$ such that $(\IT\mu_i)^{-1} (U_i) \subset U$ and $f_{|(\IT\mu_i)^{-1} (U_i)} = f_i \circ (\IT \mu_i)_{|(\IT\mu_i)^{-1} (U_i)}$. The spaces
\begin{gather*}
 \calC^\infty_{\IT M} (U) := \big\{ f \in \calC (U) \,|\, \text{$f$ is smooth}\big\}
\end{gather*}
for $U \subset \IT M$ open then form the section spaces of a sheaf $\calC^\infty_{\IT M}$ which we call the \emph{sheaf of smooth functions} on $\IT M$. By construction, the family $( \IT M_i, \IT \mu_{ij}, \IT \mu_i )$ now is a smooth projective representation of the locally ringed space $( \IT M, \calC^\infty_{\IT M} \big)$, hence $( \IT M, \calC^\infty_{\IT M} )$ becomes a prof\/inite dimensional manifold. Since $ \mu_i \circ \pi_{\IT M} = \pi_{\IT M_i} \circ \IT \mu_i $ for all $i\in \IN$, one immediately checks that the canonical map $\pi_{\IT M} \colon \IT M \rightarrow M $ is even a smooth map
between prof\/inite dimensional manifolds. With these preparations we can state:

\begin{Propanddef}\label{tang} The profinite dimensional manifold given by $( \IT M, \calC^\infty_{\IT M})$ and the pfd structure $[( \IT M_i , \IT \mu_{ij} , \IT \mu_i)]$ is called the \emph{tangent bundle} of $M$, and $\pi_{\IT M}\colon \IT M \rightarrow M$ its \emph{canonical projection}. The pfd structure $[( \IT M_i , \IT \mu_{ij} , \IT \mu_i )]$ depends only on the equivalence class $[( M_i , \mu_{ij},\mu_i)]$.
\end{Propanddef}

\begin{proof} In order to check the last statement, consider a smooth projective representation \linebreak $( M_a', \mu_{ab}' , \mu_a' )$ which is equivalent to $( M_i , \mu_{ij},\mu_i)$. Choose an equivalence transformation of smooth projective representations
 \begin{gather*}
( \varphi , F_a )\colon \ ( M_i , \mu_{ij}) \longrightarrow ( M_a', \mu_{ab}'),\qquad ( \psi , G_i)\colon \ ( M_a', \mu_{ab}') \longrightarrow ( M_i , \mu_{ij} ) .
 \end{gather*}
 Then one obtains surjective submersions
\begin{gather*}
( \varphi , \IT F_a )\colon \ ( \IT M_i , \IT \mu_{ij}) \longrightarrow ( \IT M_a', \IT \mu_{ab}'), \qquad
( \psi , \IT G_i)\colon \ ( \IT M_a', \IT \mu_{ab}') \longrightarrow ( \IT M_i , \IT \mu_{ij} )
\end{gather*}
 such that the following diagrams commute for all $i,a\in \IN$:
 \begin{gather*}
 \xymatrix{ \IT M_i && \ar[ll]_{\IT \mu_{i\,\varphi(\psi(i))} } \ar[dl]^{\IT F_{\psi (i)}} \IT M_{\varphi(\psi(i))} \\
 & \IT M_{\psi (i)}' \ar[ul]^{\IT G_i} } \qquad \text{and} \qquad
 \xymatrix{ \IT M_a' && \ar[ll]_{\IT \mu_{a\,\psi( \varphi(a))}' } \ar[dl]^{\IT G_{\varphi(a)}} \IT M_{\psi(\varphi(a))}' \\
 & \IT M_{\varphi (a)} \ar[ul]^{\IT F_a}
 }\end{gather*}
 Hence, $( \IT M_a', \IT \mu_{ab}')$ is a smooth projective system which is equivalent to $( \IT M_i, \IT \mu_{ij})$. Now recall that the map $\IT \mu_a' \colon \IT M \rightarrow \IT M_a'$ is def\/ined by $\IT \mu_a' ( Z_p) = Z_p \circ (\mu_a')^*$, where $Z_p\in \IT_pM $, $p\in M$, and $(\mu_a')^*$ denotes the pullback by $\mu_a'$. One concludes that for all $i\in \IN$
 \begin{gather*}
 \IT G_i \circ \IT \mu_{\psi (i)}' (Z_p) = \IT G_i \big( Z_p \circ (\mu_{\psi (i)}')^*\big) = Z_p \circ (\mu_{\psi (i)}')^* \circ G_i^* = Z_p \circ \mu_i^* = \IT \mu_i (Z_p) ,
 \end{gather*}
and likewise that $ \IT F_a \circ \IT \mu_{\varphi (a)} (Y_p) = \IT \mu_a' (Y_p) $ for all $a\in \IN$. This entails that the smooth projective representations $(\IT M_i, \IT \mu_{ij}, \IT \mu_i)$ and $( \IT M_a', \IT \mu_{ab}' , \IT \mu_a')$ of the tangent bundle $(\IT M, \calC^\infty_{\IT M})$ are equivalent, and the proof is f\/inished.
\end{proof}

\begin{Remark}\quad
\begin{enumerate}\itemsep=0pt
\item[a)] By Example~\ref{Ex:pfdmfd}(c), the induced smooth projective system $( \IT M_i, \IT \mu_{ij})$ has the canonical smooth projective limit
\begin{gather*}
 \big( \widetilde{\IT} M ,\calC^\infty_{\widetilde{\IT} M} \big):= \lim \limits_{\longleftarrow \atop i \in \IN} \big(\IT M_i,\calC^\infty_{\IT M_i}\big).
\end{gather*}
Denote its canonical maps by $\widetilde{\IT}\mu_i \colon \widetilde{\IT} M \rightarrow \IT M_i$. By the universal property of projective limits there exists a unique smooth map
\begin{gather*}
 \tau \colon \ \IT M \longrightarrow \widetilde{\IT} M
\end{gather*}
such that $\widetilde{\IT} \mu_i \circ \tau = \IT \mu_i$ for all $i\in \IN$. By construction of the prof\/inite dimensional manifold structure on the tangent bundle $\IT M$, the map $\tau$ is even a linear dif\/feomorphism, and is in fact given by
\begin{gather*}
 \IT M \ni Y \longmapsto \big( \IT \mu_i (Y) \big)_{i\in \IN} \in \widetilde{\IT}M.
\end{gather*}

\item[b)] As a generalization of the tangent bundle, one can def\/ine for every $k\in \IN^* := \IN \setminus \{ 0 \}$ the tensor bundle $\IT^{k,0}M$ of $M$. First, one puts for every $p\in M$
 \begin{gather*}
 \IT^{k,0}_pM := \widehat{\bigotimes}^k\IT_p M,
 \end{gather*}
where $\widehat{\otimes}$ denotes the completed projective tensor product, see Example~\ref{ExTensorProds}(a) or \cite{GroPTTEN,TreTVSDK}. The canonical maps $\IT\mu_{p,i} := {\IT\mu_i}_ {|\IT_pM}\colon \IT_p M \rightarrow \IT_{p_i} M_i$, $p_i := \mu_i(p)$ induce continuous linear maps
 \begin{gather*}
 \IT^{k,0}\mu_{p,i} := \widehat{\bigotimes}^k\IT\mu_{p,i} \colon \ \IT^{k,0}_pM\longrightarrow \IT^{k,0}_{p_i} M_i
 \end{gather*}
by the universal property of the completed projective tensor product. Likewise, one constructs for $i\leq j$ the continuous linear maps
 \begin{gather*}
 \IT^{k,0}\mu_{p_j,ij} \colon \ \IT^{k,0}_{p_j}M_j\longrightarrow \IT^{k,0}_{p_i} M_i,
 \end{gather*}
 which turn $\big( \IT^{k,0}_{p_i} M_i , \IT^{k,0}\mu_{p_j,ij} \big) $ into a projective system of (f\/inite dimensional) real vector spaces. By Theorem~\ref{ThmAppProjLimTensProd}, its projective limit within the category of locally convex topo\-lo\-gi\-cal Hausdorf\/f spaces is given by $\IT^{k,0}_pM$ together with the continuous linear maps $\IT^{k,0}\mu_{p,i}$, that means we have
 \begin{gather*} 
 \IT^{k,0}_pM = \lim\limits_{\longleftarrow \atop i\in \IN} \IT^{k,0}_{p_i} M_i .
 \end{gather*}
 Now def\/ine
 \begin{gather*}
 \IT^{k,0}M:=\bigcup_{p\in M} \IT^{k,0}_p M,
 \end{gather*}
 and give $\IT^{k,0}M$ the coarsest topology such that all the canonical maps
 \begin{align*}
 \IT^{k,0}\mu_i\colon \ & \IT^{k,0}M \longrightarrow \IT^{k,0}M_i, \\
 & Z_1 \otimes \cdots \otimes Z_k \longmapsto \IT \mu_i (Z_1) \otimes \cdots \otimes \IT \mu_i (Z_k)
 \end{align*}
are continuous. By construction, $\IT^{k,0}M$ together with the maps $\IT^{k,0}\mu_i$ has to be a projective limit of the projective system $\big( \IT^{k,0}M_i,\IT^{k,0}\mu_{ij} \big)$. The sheaf of smooth functions $\calC^\infty_{\IT^{k,0}M}$ is uniquely determined by requiring axiom~\ref{Ite:StrucSheaf} to hold true. One thus obtains a prof\/inite dimensional manifold which depends only on the equivalence class of the smooth projective representation and which will be denoted by $ \IT^{k,0}M$ in the following. Moreover, $\IT^{k,0}$ even becomes a functor on the category of prof\/inite dimensional manifolds. If $(N,\calC^\infty_{N})$ is another prof\/inite dimensional manifold and $f\colon M\to N$ a smooth map, then one naturally obtains the smooth map
 \begin{align*}
 \IT^{k,0}f\colon \ & \IT^{k,0}M\longrightarrow \IT^{k,0}N, \\
 &Z_1 \otimes \cdots \otimes Z_k \longmapsto \IT f (Z_1) \otimes \cdots \otimes \IT f (Z_k),
\end{align*}
 which satisf\/ies $\pi_{\IT^{k,0}N}\circ \IT^{k,0}f = f \circ \pi_{\IT^{k,0}M}$.
\end{enumerate}
\end{Remark}

We continue with:

\begin{Definition}
Let $U\subset M$ be open. Then a smooth section $V\colon U \rightarrow \IT M$ of $\pi_{\IT M}\colon \IT M \rightarrow M$ is called a~\emph{smooth vector field} on $M$ over $U$. The space of smooth vector f\/ields over $U$ will be denoted by $\calX^\infty (U)$.
\end{Definition}

Assume that for $U\subset M$ open we are given a smooth vector f\/ield $V\colon U\rightarrow \IT M$ and a~smooth function $f\colon U\rightarrow \IR$. We then def\/ine a function $Vf$ over $U$ by putting for $p\in U$
 \begin{gather*} 
 Vf (p) := V(p) \big( [f]_p \big) .
 \end{gather*}
 \begin{Lemma}
 For every $V\in \calX^\infty (U)$ and $f\in \calC^\infty (U)$, the function $Vf$ is smooth.
 \end{Lemma}
 \begin{proof} Choose a point $p\in U$, and then an open $U_i\subset M_i$ and a function $f_i \in \calC^\infty (U_i)$ for some appropriate $i\in \IN$ such that $p\in \mu_i^{-1} (U_i) \subset U$ and
 \begin{gather} \label{Eq:LocRepFunc}
 f_{|\mu_i^{-1} (U_i)} = f_i \circ {\mu_i}_{|\mu_i^{-1} (U_i)} .
 \end{gather}
Consider $V_i\colon M \rightarrow \IT M_i$, $V_i := \IT \mu_i \circ V$. Since $V_i$ takes values in a f\/inite dimensional smooth manifold, there exists an integer $j_p\geq i$ (which we brief\/ly denote by $j$, if no confusion can arise), an open $U_{pj}\subset M_j$ and a smooth vector f\/ield $V_p \colon U_{pj} \rightarrow \IT M_i$ along $ \mu_{ij}$ such that $p\in \mu_j^{-1} (U_{pj})$, $ U_{pj} \subset \mu_{ij}^{-1} (U_i)$ and
 \begin{gather*} 
 {\IT \mu_i \circ V}_{|\mu_j^{-1} (U_{pj})} = V_p \circ {\mu_j}_{|\mu_j^{-1} (U_{pj})} .
 \end{gather*}
 Now def\/ine $g_{pj} \colon U_{pj} \rightarrow \IR$ by
 \begin{gather*}
 g_{pj} (q_j) := V_p (q_j) \big( [f_i]_{\mu_{ij}(q_j)} \big) \qquad \text{for all} \quad q_j \in U_{pj}.
 \end{gather*}
 Then $g_{pj}$ is smooth, hence $g_p := g_{pj} \circ {\mu_j}_{| \mu_j^{-1} (U_{pj})}$ is an element of $\calC^\infty \big(\mu_j^{-1} (U_{pj}) \big)$.
 Now one checks for $q\in \mu_j^{-1} (U_{pj})$ that
 \begin{gather*} 
 g_p (q) = V_p (\mu_j(q)) \big( [f_i]_{\mu_i(q)} \big) = V_i (q) \big( [f_i]_{\mu_i(q)} \big) = V (q) \big( [f]_q \big)
 \end{gather*}
 by equation~\eqref{Eq:LocRepFunc}. Hence
 \begin{gather*}
 g_p = (Vf)_{|\mu_j^{-1} (U_{pj})} ,
 \end{gather*}
 and $Vf$ is smooth indeed.
 \end{proof}

\begin{Proposition} Every vector field $V\in \calX^\infty (U)$ defined over an open subset $U\subset M$ induces a derivation
 \begin{gather*}
 \delta_V \colon \ \calC^\infty (U) \longrightarrow\calC^\infty (U),\qquad f \longmapsto Vf.
 \end{gather*}
\end{Proposition}
\begin{proof}
 By construction, it is clear that the map
 \begin{gather*}
 \garb (U)\ni f \longmapsto Vf \in \garb (U)
 \end{gather*}
 is $\IR$-linear. It remains to check that $\delta_V $ is a derivation, or in other words that it satisf\/ies Leibniz' rule. But this follows immediately by the def\/inition of the action of~$V$ on $\garb (U)$ and the fact that $V(p) \in \operatorname{Der} \big(\calC^\infty_p,\IR\big)$ for all $p\in U$. More precisely, one has, for $p\in U$ and $f,g \in \garb (U)$,
 \begin{gather*}
 V(fg) (p) = V(p) \big( [fg]_p \big) = f(p) V(p) \big( [g]_p \big) + g(p) V(p) \big( [f]_p \big) = \big( f V(g) + g V(f) \big) (p).
 \end{gather*}
 This f\/inishes the proof.
\end{proof}

\begin{Definition}Let $U\subset M$ be open. A smooth vector f\/ield $V \in \calX^\infty (U)$ is called \emph{local}, if for every $i\in \IN$ there is an integer $m_i\geq i$ and a~smooth vector f\/ield $V_{im_i} \colon \mu_{m_i} (U) \rightarrow \IT M_i$ along~$ \mu_{im_i}$ such that
 \begin{gather} \label{typ}
 \IT \mu_i \circ V = V_{im_i} \circ {\mu_{m_i}}_{|U} .
 \end{gather}
The space of local vector f\/ields over $U$ will be denoted by $\locX (U)$.
\end{Definition}

\begin{Remark}{\samepage\quad
\begin{enumerate}\itemsep=0pt
\item[a)] Obviously, $\calX^\infty$ is a sheaf of $\calC^\infty$-modules on $M$, and $\locX $ a presheaf of $\locC$-modules. Note that $\locX $ depends only on the pfd structure $[(M_i,\mu_{ij},\mu_i)]$.
\item[b)] Let $V\in \locX (U)$, and pick a representative $(M_i,\mu_{ij},\mu_i)$ of the underlying pfd structure. If $(m_i)_{i\in \IN}$ is a sequence of integers such that~(\ref{typ}) holds true, we sometimes say that $V$ is of \emph{type $(m_0,m_1,m_2, \ldots )$ with respect to the smooth projective representation $(M_i,\mu_{ij},\mu_i)$}. The notion of the type of a local vector f\/ield is known from jet bundle literature~\cite{anderson}, where it makes perfect sense, since the prof\/inite dimensional manifold of inf\/inite jets has a distinguished representative of the underlying pfd structure, see Section~\ref{inf}.
\end{enumerate}}
\end{Remark}

Now we are in the position to prove the following structure theorem:

\begin{Theorem}\label{deri} The map
 \begin{gather*}
 \delta\colon \ \calX^\infty (M) \longrightarrow \operatorname{Der} \big( \calC^\infty (M), \calC^\infty (M) \big),\qquad V\longmapsto \delta_V
 \end{gather*}
 is a bijection. Moreover, for every $V\in \calX^\infty(M)$, the derivation $\delta_V$ leaves the algebra $\locC (M)$ of local functions on $M$ invariant, if and only if one has $V \in \locX (M) $.
\end{Theorem}

\begin{proof} \emph{Surjectivity}: Assume that $D\colon \calC^\infty (M) \rightarrow \calC^\infty (M)$ is a derivation. Then one obtains for each
 $i\in \IN$ and point $p\in M$ a linear map
 \begin{gather*}
 D_{pi} \colon \ \calC^\infty (M_i) \longrightarrow \IR, \qquad f \longmapsto D (f \circ \mu_i) (p) .
 \end{gather*}
 Note that for $f, f' \in \calC^\infty (M_i)$
 \begin{align*}
 D_{pi} (ff')& = D \big( (f f') \circ \mu_i \big) (p) = f \circ \mu_i (p) D (f'\circ \mu_i) (p) + f' \circ \mu_i (p) D (f \circ \mu_i) (p)\\
 & = f \circ \mu_i (p) D_{pi} (f') + f' \circ \mu_i (p) D_{pi} (f ) ,
 \end{align*}
 which entails that there is a tangent vector $V_{pi} \in \IT_{\mu_i(p)}M_i$ such that $ D_{pi} = V_{pi}$. Observe that for $j\geq i$ the relation
 \begin{gather*}
 D_{pi} (f) = D (f\circ \mu_i) (p) = D (f\circ \mu_{ij} \circ \mu_j) (p) = D_{pj} (f\circ \mu_{ij})
 \end{gather*}
holds true, which entails that $V_{pi} = \IT \mu_{ij} \circ V_{pj}$. Hence, the sequence of tangent vectors $(V_{pi})_{i\in \IN}$ def\/ines an element~$V_p$ in
 \begin{gather*}
 \IT_pM \cong \lim\limits_{\longleftarrow \atop i\in \IN} \IT_{\mu_i (p)} M_i .
 \end{gather*}
 We thus obtain a section $V\colon M\to\IT M$, $p \mapsto V_p$. Let us show that $V$ is smooth. To this end, consider the composition
 \begin{gather*}
 V_i := \IT \mu_i \circ V\colon \ M \longrightarrow \IT M_i.
 \end{gather*}
By construction $V_i (p) = V_{pi}$ for all $p\in M$. It suf\/f\/ices to show that each of the maps $V_i$ is smooth. To show this, choose a coordinate neighborhood $U_i\subset M_i$ of $\mu_i(p)$, and coordinates
 \begin{gather*}
 \big(x^1,\dots ,x^k\big) \colon \ U_i \longrightarrow \IR^k.
 \end{gather*}
 Then
 \begin{gather*}
\big(x^1\circ \pi_{\IT U_i},\dots ,x^k\circ \pi_{\IT U_i}, \Id x^1, \dots , \Id x^k\big) \colon \ \IT U_i \longrightarrow \IR^{2k}
 \end{gather*}
is a local coordinate system of $\IT M_i$. The map $V_i$ now is proven to be smooth, if $\Id x^l\circ V_i$ is smooth for $1\leq l\leq k$. But
 \begin{gather*}
 \Id x^l \circ {V_i}_{|\mu_i^{-1}(U_i)} = D\big(x^l\circ {\mu_i}_{|\mu_i^{-1}(U_i)}\big),
 \end{gather*}
 since for $q\in \mu_i^{-1} (U_i)$
 \begin{gather*}
 \Id x^l \circ {V_i}_{|\mu_i^{-1}(U_i)}(q) = V_{qi}(q) \big([x^l]_{\mu_i(q)}\big)= D_{qi} \big(x^l\big) = D\big(x^l\circ {\mu_i}_{|\mu_i^{-1}(U_i)}\big)(q).
 \end{gather*}
Hence each $V_i$ is smooth, and $V $ is a smooth vector f\/ield on $M$ which satisf\/ies $\delta_V=D$. This proves surjectivity.

\emph{Injectivity}: Assume that $V$ is a smooth vector f\/ield on $M$ such that $\delta_V=0$. This means that $\delta_Vf (p)=0$ for all $f\in \calC^\infty (M)$ and $p\in M$. Choose now a $i\in\IN$ and let $f_i$ be a smooth function on $M_i$. Put $f:=f_i\circ \mu_i$ and $V_i =\IT \mu_i \circ V$. Then, we have for all $p\in M$
 \begin{gather*}
 V_i (p) \big( [f_i]_{\mu_i(p)} \big) = \delta_V f (p) = 0,
 \end{gather*}
which implies that $V_i(p) =0$ for all $p\in M$. Since $V(p)$ is the projective limit of the $V_i(p)$, we obtain $V(p)=0$ for all $p\in M$, hence $V=0$. This f\/inishes the proof that $\delta$ is bijective.

\emph{Local vector fields}: Next, let us show that for a local vector f\/ield $V \colon M \rightarrow \IT M$ the deriva\-tion~$\delta_V$ maps local functions to local ones. To this end choose for every $i\in \IN$ an integer $m_i\geq i$ such that there exists a smooth vector f\/ield $V_{im_i} \colon M_{m_i} \rightarrow \IT M_i$ along $\mu_{im_i}$ which satisf\/ies
 \begin{gather*}
 \IT \mu_i \circ V = V_{im_i} \circ \mu_{m_i} .
 \end{gather*}
Now let $f$ be a local function on $M$, which means that $f= f_i \circ \mu_i$ for some $i\in \IN$ and $f_i \in \calC^\infty (M_i )$. Def\/ine $g_{m_i}\in \calC^\infty (M_{m_i})$ by $g_{m_i} (q) = V_{im_i} (q) \big( [ f_i ]_{\mu_{im_i} (q)} \big)$ for all $q\in M_{m_i}$. Then, one obtains for $p\in M$
 \begin{gather*}
 \delta_V f (p) = V_{im_i} (\mu_{m_i} (p)) \big( [ f_i ]_{\mu_i (p)} \big) = g_{m_i}(\mu_{m_i} (p)) ,
 \end{gather*}
 which means that $\delta_V f = g_{m_i} \circ \mu_{m_i}$ is local.

\emph{Invariance of $\locC (M)$}: Finally, we have to show that if $\delta_V$ for $V \in \calX^\infty (M)$ leaves the space $\locC (M) $ invariant, the vector f\/ield $V$ has to be local. To this end f\/ix $i\in \IN$ and
 choose a proper embedding
 \begin{gather*}
 \chi=(\chi_1 , \ldots , \chi_N) \colon \ M_i \longhookrightarrow \IR^N.
 \end{gather*}
Then $\chi_l \circ \mu_i\in \locC (M)$ for $l=1,\ldots,N$, hence there exist by assumption $j_1, \ldots , j_N\in \IN$ and $g_{il} \in \calC^\infty (M_{j_l}) $ such that
 \begin{gather*}
 \delta_V( \chi_l \circ \mu_i) = g_{il} \circ \mu_{j_l} .
 \end{gather*}
 After possibly increasing the $j_l$, we can assume that $m_i:= j_1 = \dots = j_N \geq i$. Denote by $z_l\colon \IR^N \rightarrow \IR$ the canonical projection onto the $l$-th coordinate, and def\/ine the vector f\/ield $\widetilde {V}_{im_i} \colon M_{m_i} \rightarrow \IT \IR^N$ along $\chi \circ \mu_{im_i}$ by
 \begin{gather*}
 \widetilde {V}_{im_i} (q) := \sum_{l=1}^N g_{lj_l} (q) \dfrac{\partial}{\partial z_l}_{|\chi (\mu_{im_i} (q))} \qquad \text{for} \quad q \in M_{m_i}.
 \end{gather*}
 Since by construction
 \begin{gather*}
 \widetilde {V}_{im_i} (\mu_{m_i} (p)) \big( [z_l]_{\chi (\mu_i (p))} \big) = g_{il}(\mu_{m_i} (p)) =
 (\IT \mu_i \circ V) (p) \big( [\chi_l]_{\mu_i (p)} \big)
 \end{gather*}
 for all $p\in M$, $\widetilde {V}_{im_i} (y)$ is in the image of $\IT_q\chi$ for every $q\in M_{m_i}$, hence
 \begin{gather*}
 V_{im_i} \colon \ M_{m_i} \longrightarrow \IT M_i, \qquad q \longmapsto (\IT_y\chi)^{-1} \big( \widetilde {V}_{im_i} (q) \big)
 \end{gather*}
 is well-def\/ined and satisf\/ies $\IT \mu_i\circ V = V_{im_i} \circ \mu_{m_i}$. Therefore, $V$ is a local vector f\/ield.
\end{proof}

The following result is an immediate consequence of Theorem \ref{deri}.

\begin{Corollary}For all $V,W\in \calX^\infty(M)$, the map
 \begin{gather*}
 [V,W] \colon \ \calC^\infty (M) \longrightarrow \calC^\infty (M), \qquad f \longmapsto V (Wf) - W (Vf)
 \end{gather*}
is a derivation on $\calC^\infty (M)$. Its corresponding underlying vector field will be denoted by $[V,W]$ as well, and will be called the \emph{Lie bracket} of~$V$ and~$W$. The Lie bracket of vector fields turns~$\calX^\infty(M)$ into a~Lie algebra.
\end{Corollary}

\begin{Remark}
 Unlike in the f\/inite dimensional case, smooth vector f\/ields on a prof\/inite dimensional manifold of inf\/inite dimension need not be integrable. See \cite[Section~3 \& Concluding comments]{chetverikov} and references therein for further information on this phenomenon and an integrability criterion in the particular case of the solution manifold of a formally integrable PDE.
\end{Remark}

\subsection{Dif\/ferential forms}
In the f\/inite dimensional case, dif\/ferential forms are usually def\/ined as smooth sections of alter\-nating powers of the cotangent bundle. This approach can not directly be carried over to the prof\/inite dimensional case, since there is no canonical construction of the cotangent bundle of a~prof\/inite dimensional manifold. The reason for this is that unlike for the tangent bundle functor, which is a covariant functor and transforms a smooth projective system $(M_i,\mu_{ij})$ into a~projective system $(TM_i,T\mu_{ij})$, forming the cotangent bundle is neither covariant nor contravariant functorial on the category of f\/inite dimensional smooth manifolds and smooth maps. The way out is to def\/ine dif\/ferential forms on prof\/inite dimensional manifolds as continuous maps on a~tensor product of the tangent bundle such that these maps are locally pull-backs of dif\/ferential forms on the components of a~smooth projective representation.

\begin{Definition}
Let $k\in \IN$ and $U \subset M$ open.
\begin{enumerate}\itemsep=0pt
\item[a)]
 A continuous map
\begin{gather*}
\omega\colon \ (\pi_{\IT^{k,0}M})^{-1} (U) \longrightarrow \IR
\end{gather*}
is called a \emph{differential form of order $k$} or a \emph{$k$-form} on $M$ over $U$, if for every point $p\in U$ there is some $i\in \IN$, an open subset $U_i\subset M_i$ with $p\in \mu_i^{-1}(U_i) \subset U$ and a $k$-form $\omega_i \in \Omega^k (U_i)$ such that
 \begin{gather*}
 \omega_{|( \mu_i \circ \pi_{\IT^{k,0}M})^{-1} (U_i) } = \omega_i \circ {\IT^{k,0}\mu_i}_{|( \mu_i \circ \pi_{\IT^{k,0}M})^{-1} (U_i) } .
 \end{gather*}
 More precisely, this means that for all $y \in \mu_i^{-1} (U_i) $, and $V_1, \ldots,V_k\in \pi^{-1} (y)\subset \IT M $ the relation
 \begin{gather*}
 \omega ( V_1 \otimes \cdots \otimes V_k ) = \omega_i \big( \IT \mu_i (V_1) \otimes \cdots \otimes \IT \mu_i (V_k)\big)
 \end{gather*}
holds true. In particular, a $k$-form $\omega$ over $U$ is antisymmetric and $k$-multilinear in its arguments. The space of $k$-forms over $U$ will be denoted by $\Omega^k (U)$.

\item[b)] A $k$-form $\omega \in \Omega^k (U)$ is called \emph{local}, if there is an open $U_i \subset M_i$ for some $i\in \IN$ and a~$k$-form $ \omega_i \in \Omega^k (U_i)$ such that $U \subset \mu_i^{-1} (U_i)$ and $\omega = (\mu_i^* \omega_i)_{|U}$, where here and from now on we use the notation $\mu_i^* \omega_i$ for the form $\omega_i \circ {\IT^{k,0} \mu_i}$. The space of local $k$-forms over~$U$ will be denoted by~$\lockOmega (U)$.
\end{enumerate}
 \end{Definition}

\begin{Remark}\quad{\samepage
\begin{enumerate}\itemsep=0pt
\item[a)] By a straightforward argument one checks that the spaces $\Omega^k (U)$ and $\lockOmega (U)$ only depend on the pfd structure $[(M_i,\mu_{ij},\mu_i)]$. Moreover, for every representative $(M_i,\mu_{ij},\mu_i)$ of the pfd structure, $\lockOmega (U)$ together with the family of pull-back maps $\mu_i^* \colon \Omega^k (\mu_i (U)) \rightarrow \lockOmega (U) $ is an inductive limit of the injective system of linear spaces
$\big( \Omega^k (\mu_i (U) ) , \mu_{ij}^* \big)_{i\in \IN}$.

\item[b)]
 By construction, it is clear that $\Omega^k$ forms a sheaf of $\mathscr{C}^{\infty}$-modules on~$M$ and $\lockOmega$ a presheaf of $\mathscr{C}^{\infty}_{\mathrm{loc}}$-modules. Moreover, $\Omega^k$ coincides with the sheaf associated to~$\lockOmega$.

\item[c)] The representative $\mathcal{M}:=(M_i,\mu_{ij},{\mu}_j)$ of the pfd structure on~$M$ leads to the particular f\/iltration $\mathcal{F}^{\mathcal{M}}_\bullet$ of the presheaf $\locOmega{k}$ of local $k$-forms on $M$ by putting, for $l\in \IN$,
 \begin{gather*}
 \mathcal{F}^{\mathcal{M}}_l\big(\locOmega{k}\big) := \mu_l^*\Omega^k_{M_l} .
 \end{gather*}
 Observe that this f\/iltration has the property that
 \begin{gather*}
 \locOmega{k} = \bigcup_{l\in\IN}\mathcal{F}^{\mathcal{M}}_l\big(\locOmega{k}\big).
 \end{gather*}
\end{enumerate}}
\end{Remark}

\begin{Propanddef}\quad
\begin{enumerate}\itemsep=0pt
\item[$a)$] There exists a uniquely determined morphism of sheaves $\Id\colon \Omega^k \rightarrow \Omega^{k+1}$ such that
\begin{gather*}
\Id ( \mu_i^* \omega_i ) = \mu_i^* ( \Id \omega_i )\qquad \text{for all} \quad i\in \IN, \quad U_i\subset M_i \ \text{open}, \quad \omega_i\in\Omega^k(U_i).
\end{gather*}
The morphism $\Id$ is called the \emph{exterior derivative}, fulfills $\Id \circ \Id =0$, and maps $\lockOmega$ to $\locOmega{k+1}$.

\item[$b)$] There exists a uniquely determined morphism of sheaves
 \begin{gather*}
\wedge\colon \ \Omega^k \times \Omega^l \longrightarrow \Omega^{k+l},
\end{gather*}
called the \emph{wedge product}, such that for all $i\in \IN$, $U_i\subset M_i$ open, $\omega_i\in\Omega^k(U_i)$, and $\mu_i\in\Omega^l(U_i)$ one has
 \begin{gather*}
 \mu_i^* \omega_i \wedge \mu_i^* \mu_i = \mu_i^* ( \omega_i \wedge \mu_i ).
 \end{gather*}
The wedge product also leaves $\locOmega{\bullet}$ invariant.

\item[$c)$] Given a vector field $V \in \calX^\infty (M)$, there exists the \emph{contraction with $V$} that means the sheaf morphism
 \begin{gather*}
 i_V \colon \ \Omega^k \longrightarrow \Omega^{k-1},
 \end{gather*}
 which is uniquely determined by the requirement that for all $\omega \in \Omega^k (U)$ with $U\subset M$ open, $p\in U$, and $W_1,\ldots , W_{k-1} \in \IT_pM$ the relation
 \begin{gather*}
 i_V (\omega) ( W_1 \otimes \cdots \otimes W_{k-1} ) = \omega \big( V(p)\otimes W_1 \otimes \dots \otimes W_{k-1}\big)
 \end{gather*}
 holds true. If $V$ is a local vector field, contraction with $V$ leaves $\locOmega{\bullet}$ invariant.
\end{enumerate}
\end{Propanddef}

\begin{proof}
 Using the sheaf property of $\Omega^k$ one can reduce the claims to local statements which are immediately proved.
\end{proof}

\section{Jet bundles and formal solutions of nonlinear PDEs}\label{fsnlpde}

The aim of this section is to develop a precise geometric notion of formally integrable (systems of) partial dif\/ferential equations, and to show that the formal solution spaces of these equations canonically become a prof\/inite dimensional manifold in the sense of Section~\ref{pfd}. Finally, we are going to give a criterion for the formal integrability of nonlinear scalar partial dif\/ferential equations, and apply this result to a class of interacting relativistic scalar f\/ield theories that arise in theoretical physics.

We refer the reader to \cite{krasvino,gia,KraLycVinGJSNPDE, saunders} and also to~\cite{pomm,seiler} for introductionary texts on jet bundles, where the latter two references have a strong focus on the highly nontrivial algorithmic aspects of this theory. A nice short overview is also included in the introduction of~\cite{vita}.

\subsection{Finite order jet bundles}\label{finjet}

For the rest of the paper, we f\/ix a f\/iber bundle $\pi\colon E\to X$. Moreover, $F$ will denote the typical f\/iber of $\pi$ and we set $m:=\dim X$, $n:=\dim F$.

Then one has $\dim E = m+n$ and the f\/ibers $\pi^{-1}(p)\subset E $ become $n$-dimensional submanifolds, which are dif\/feomorphic to~$F$. There are distinguished charts for $E$:

\begin{Definition} A manifold chart $(x,u)\colon W\to \IR^m\times \IR^n$ of $E$ def\/ined over some open $W\subset E$ is called a \emph{fibered} \emph{chart of $\pi$}, if for all $e,e'\in W$ with $\pi(e)=\pi(e')$ the equality $x(e)=x(e')$ holds true.
\end{Definition}

\begin{Remark}\label{ReCha}\quad
\begin{enumerate}\itemsep=0pt
\item[a)] Sometimes, f\/ibered charts are called \emph{adapted charts}.

\item[b)] Note that a f\/ibered chart $(x,u)\colon W\to \IR^m\times \IR^n$ for $\pi$ canonically gives rise to a well-def\/ined manifold chart on $X$. It is given by
\begin{gather}\label{st}
\tilde{x}\colon \ \pi(W)\longrightarrow \IR^m,\qquad p\longmapsto x(e),
\end{gather}
where $e\in W\cap \pi^{-1}(p)$ is arbitrary.

\item[c)] On the other hand, a manifold atlas for $E$ that consists of f\/ibered charts for $\pi$ can be constructed from manifold charts for $X$ and from the local triviality of $E$ as follows: For an arbitrary $e\in E$, take a bundle chart $\phi\colon \pi^{-1}(U)\to U \times F$ around $\pi(e)$, that is, $U$ is an open neighbourhood of $\pi(e)$ and $\phi\colon \pi^{-1}(U)\to U \times F$ is a dif\/feomorphism such that
\begin{gather}\label{dia}
\begin{split}&
\xymatrix{
\pi^{-1}(U) \ar[d]_{\pi} \ar[r]^{\phi} & U\times B \ar[dl]^{\mathrm{pr}_1}\\
U
}\end{split}
\end{gather}
commutes. Let $\tilde{x}\colon U\to\IR^m$ be a manifold chart of $X$ (here we assume that $U$ is small enough), and let $\tilde{u}\colon B\to\IR^n$ be a~manifold chart of $F$. Then
\begin{gather*}
(\tilde{x}\circ \pi,\tilde{u}\circ\mathrm{pr}_2\circ \phi)=(\tilde{x}\circ \mathrm{pr}_1\circ \phi,\tilde{u}\circ\mathrm{pr}_2\circ \phi)\colon \ \pi^{-1}(U)\longrightarrow \IR^m\times \IR^n
\end{gather*}
is a f\/ibered chart of $\pi$. Note here that by the commutativity of~(\ref{dia}), the notation ``$\tilde{x}$'' is consistent with~(\ref{st}).
\end{enumerate}
\end{Remark}

Let us introduce the following notation for multi-indices, which will be convenient in the following: For any $k_1,k_2\in\IN$ with $k_1\leq k_2$ let $\IN^m_{k_1,k_2}$ denote the set of all multi-indices $I\in \IN^m$ such that
\begin{gather*}
k_1\leq |I|:=\sum^m_{j=1}I_j\leq k_2
\end{gather*}
and let $\jetfib(m,k_1,k_2)$ denote the linear space of all maps $\IN^m_{k_1,k_2}\to \IR$. For any $l\leq m$, $i_1,\dots,i_l\in \{1,\dots,m\}$, the symbol $1_{i_1\dots i_l}\in\IN^m$ will denote the multi-index which has a $1$ in its $i_j$'s slot for $j=1,\dots,l$, and a $0$ elsewhere.

Any $\psi \in \Gamma^{\infty} (p;\pi)$ allows the following local description: Choose a f\/ibered chart $(x,u)\colon W\to\IR^m\times\IR^n$ of $\pi$ with $W\cap \pi^{-1}(p)\ne \varnothing$. Then one has $x\circ \psi = \tilde{x}$ near $p$, so that $\psi$ is determined by the coordinates $(u^1\circ \psi,\dots,u^n\circ \psi)=u\circ \psi$ near $p$. The special form of the following def\/inition is motivated by the latter fact:

\begin{Definition}\label{equiv}
Let $p\in X$, $k\in\IN$. Any two $\psi, \varphi\in\Gamma^{\infty} (p;\pi)$ are called \emph{$k$-equivalent at $p$}, if for every f\/ibered chart $(x,u)\colon W\to\IR^m\times\IR^n$ of $\pi$ with $W\cap \pi^{-1}(p)\ne \varnothing$ one has
\begin{gather}\label{equi}
\f{\partial^{|I|} \left(u^{\alpha}\circ\psi\right)}{\partial \tilde{x}^I} (p)=\f{\partial^{|I|} \left(u^{\alpha}\circ\varphi\right)}{\partial\tilde{x}^I} (p)
\end{gather}
for all $\alpha=1,\dots,n$ and all $I\in \IN^m_{0,k}$. The corresponding equivalence class $\jet^k_p \psi$ of $\psi$ is called the \emph{$k$-jet of~$\psi$ at~$p$.}
\end{Definition}

\begin{Remark}\label{fadibru}
In fact, it is enough to check (\ref{equi}) in \emph{some} f\/ibered chart. This can be proved by induction on~$k$, using the multivariate version of Faa di Bruno's formula \cite[Theorem~2.1]{ConSavMFBFA} (details can be found in Lemma~6.2.1 in~\cite{saunders}).
\end{Remark}

Let us now come to several structures that can be def\/ined via jets. Denoting by
\begin{gather*}
 \Jet^k(\pi):= \bigcup_{p\in X}\big\{ \jet^k_p\psi\,|\, \psi \in\Gamma^{\infty} (p;\pi) \big\}
\end{gather*}
the collection of all $k$-jets in $\pi$, we obtain the surjective maps
\begin{gather*}
\pi_k\colon \ \Jet^k(\pi)\longrightarrow X,\qquad \jet^k_p\psi\longmapsto p ,\qquad
\pi_{0,k}\colon \ \Jet^k(\pi)\longrightarrow E,\qquad \jet^k_p\psi\longmapsto \psi(p).
\end{gather*}
Using these maps, one can give $\Jet^k(\pi)$ the structure of a f\/inite dimensional manifold in a canonical way: For every f\/ibered chart $(x, u)\colon W\to\IR^m\times\IR^n$ of $\pi$ and every $I\in\IN^m_{0,k}$, one def\/ines the map
\begin{align}
(x_k,u_{k,I}) \colon \ & \pi_{0,k}^{-1}(W)\longrightarrow \IR^m\times \IR^{n \dim \jetfib(m,0,k)}, \nonumber\\
&\jet^k_p\psi\longmapsto \left(\tilde{x}(p),\f{\partial^{|I|} (u^{1}\circ\psi)}{\partial \tilde{x}^I}(p),\dots,\f{\partial^{|I|}\left(u^{n}\circ\psi\right)}{\partial \tilde{x}^I}(p)\right).\label{mf1}
\end{align}

The following result is checked straightforwardly (cf.~\cite{saunders}).

\begin{Propanddef}\label{jetmf} The maps \eqref{mf1} define an $m+n \dim \jetfib(m,0,k)$-dimensional manifold structure on $\Jet^k(\pi)$. In view of this fact, $\Jet^k(\pi)$ is called the \emph{$k$-jet manifold} corresponding to~$\pi$.
\end{Propanddef}

For convenience, we set $\Jet^0(\pi):=E$ and $\jet^0_p\psi:=\psi(p)$ for any $\psi\in\Gamma^{\infty} (p;\pi)$, and $\pi_0:=\pi$. More generally, we have for any $ k_1\leq k_2$ the smooth surjective maps
\begin{gather*}
 \pi_{k_1,k_2}\colon \ \Jet^{k_2}(\pi)\longrightarrow \Jet^{k_1}(\pi),\qquad \jet^{k_2}_p\psi\longmapsto \jet^{k_1}_p\psi,
\end{gather*}
which satisfy $\pi_{k,k}=\mathrm{id}_{\Jet^k(\pi)}$, and if one also has $k_2\leq k_3$, then the following diagram commutes:
\begin{gather}\label{proj}
\begin{split}&
\begin{xy}
 \xymatrix{
 \Jet^{k_3}(\pi) \ar[rr]^{\pi_{k_2,k_3}} \ar[dd]_{\pi_{k_1,k_3}} & & \Jet^{k_2}(\pi)\ar[ddll]_{\pi_{k_1,k_2}}\ar[dd]^{\pi_{k_2}} \\ &&\\
 \Jet^{k_1}(\pi) \ar[rr]^{\pi_{k_1}} & & X
 }
\end{xy}
\end{split}
\end{gather}

Let us collect all structures underlying the above maps. Let $(x,u)\colon W\to\IR^m\times\IR^n$ be a~f\/ibered chart
of $\pi$. Then we set
\begin{gather}
\begin{split}
\left(\pi_k,u_k\right)\colon \ & \pi^{-1}_{k,0}(W)\longrightarrow \pi(W)\times \jetfib(m,0,k)^n,\\
&\jet^k_p\psi \longmapsto \left(p,\left\{\f{\partial^{|I|} \left(u\circ\psi\right)}{\partial \tilde{x}^I}(p)\right\}_{I\in\IN^m_{0,k}}\right)\\
& \hphantom{\jet^k_p\psi \longmapsto}{} =\left(p,\left\{\f{\partial^{|I|} (u^{1}\circ\psi)}{\partial \tilde{x}^I}(p)\right\}_{I\in\IN^m_{0,k}},\dots,\left\{\f{\partial^{|I|} \left(u^{n}\circ\psi\right)}{\partial\tilde{x}^I}(p)\right\}_{I\in\IN^m_{0,k}}\right),
\end{split}\label{co}
\\
\begin{split}
\left(\pi_{k_1,k_2},u_{k_1,k_2}\right)\colon \ & \pi^{-1}_{k_2,0}(W)\longrightarrow \pi^{-1}_{k_1,0}(W)\times \jetfib(m,k_1+1,k_2)^n,\\
&\jet^{k_2}_p\psi \longmapsto \left(\jet^{k_1}_p\psi,\left\{\f{\partial^{|I|} (u\circ\psi)}{\partial \tilde{x}^I}(p)\right\}_{I\in\IN^m_{k_1+1,k_2}}\right).
\end{split}\label{sds}
\end{gather}
If $\pi$ is a vector bundle, then, for every $p\in X$, the f\/iber $\pi_k^{-1}(p)$ canonically becomes a linear space through
\begin{gather*}
c_1 \big(\jet^k_p\psi\big)+c_2 \big(\jet^k_{p}\varphi\big):=\jet^k_p(c_1\psi+c_2\varphi),\qquad c_j\in\IR.
\end{gather*}
Furthermore, if $k_2=k$, $k_1=k-1$ and if $a\in \Jet^{k-1}(\pi)$, then the f\/iber $\pi_{k-1,k}^{-1}(a)$ carries a~canonical af\/f\/ine structure which is modelled on the linear space
\begin{gather}\label{vec}
 \Sym^{k}\big(\IT^*_{\pi_{k-1}(a)}X\big)\otimes \operatorname{ker}\big( {\IT\pi}_{|\pi_{0,k-1}(a)} \big).
\end{gather}
To see the latter fact, assume that $\pi_{0,k-1}(a)\in W$, let $\jet^k_{\pi_{k-1}(a)}\psi\in\pi_{k-1,k}^{-1}(a)$ and let $v$ be an element of~(\ref{vec}). Then $v$ can be uniquely expanded as
\begin{gather*}
 v=\sum_{I\in\IN^m_{k,k}}\sum^n_{\alpha=1}v^{\alpha}_I {\Id\tilde{x}_{I}}_{|\pi_{k-1}(a)}\otimes
 \f{\partial}{\partial{u}^{\alpha}}_{|\pi_{0,k-1}(a)},\qquad v^{\alpha}_I\in\IR,
\end{gather*}
where we have used the abbreviation
\begin{gather*}
\Id \tilde{x}_{I}:=(\Id \tilde{x}_{1})^{\otimes I_1}\odot \cdots\odot (\Id \tilde{x}_{m})^{\otimes I_m},
\end{gather*}
so that one can def\/ine $\jet^k_{\pi_{k-1}(a)}\psi+v\in \pi_{k-1,k}^{-1}(a)$ to be the uniquely determined element whose image under~(\ref{sds}) is given by
\begin{gather*}
\left(\jet^{k-1}_{\pi_{k-1}(a)}\psi,\left\{\f{\partial^{|I|} \left(u\circ\psi\right)}{\partial \tilde{x}^I}_{|\pi_{k-1}(a)}+v_I\right\}_{I\in\IN^m_{k,k}}\right).
\end{gather*}

With these preparations, one has:

\begin{Lemma}\label{struc}
Let $k,k_1,k_2\in\IN$ with $k_1\leq k_2$. Then the following assertions hold.
\begin{enumerate}\itemsep=0pt
\item[$a)$] The maps \eqref{co} turn $\pi_k\colon \Jet^k(\pi)\to X$ into a fiber bundle with typical fiber $\jetfib(m,0,k)^n$. If $\pi$ is a vector bundle, then so is $\pi_k$.
\item[$b)$] The maps \eqref{sds} turn $\pi_{k_1,k_2}\colon \Jet^{k_2}(\pi)\to \Jet^{k_1}(\pi)$ into a fiber bundle with typical fiber $\jetfib(m,l+1,k)^n$, and $\pi_{k-1,k}\colon \Jet^{k}(\pi)\to \Jet^{k-1}(\pi)$ becomes an affine bundle, modelled on the vector bundle
\begin{gather*}
 \pi^*_{k-1}\Sym^k\big( \pi_{\IT^*X} \big) \otimes \pi^*_{k-1,0}\Verti(\pi)\longrightarrow \Jet^{k-1}(\pi).
\end{gather*}
\end{enumerate}
\end{Lemma}

\begin{proof}
 The reader can f\/ind a detailed proof of Lemma~\ref{struc} in Chapter~6 of~\cite{saunders}.
\end{proof}

We close this section with a simple observation about distinguished elements of $\Gamma^{\infty}(\pi_k)$. Let $U\subset X$ be an open subset for the moment. Then for any $\psi\in \Gamma^{\infty}(U;\pi)$, the map $U\to \Jet^k(\pi)$, $p\mapsto \jet^k_p \psi$, def\/ines an element of $\Gamma^{\infty}(U;\pi_k)$, called the \emph{$k$-jet prolongation of $\psi$.} In fact, this construction induces a morphism of sheaves $\jet^k\colon \Gamma^{\infty}(\pi)\rightarrow \Gamma^{\infty}(\pi_k)$ (with values in the category of sets) such that
\begin{gather}\label{pr1}
 \pi_{k_1,k_2}\circ \jet^{k_2} =\jet^{k_1}\qquad \text{for} \quad k_1\leq k_2.
\end{gather}
It should be noted that it is not possible to write an arbitrary element of $\Gamma^{\infty}(U;\pi_k)$ as $\jet^k_U\psi$ for some $\psi\in \Gamma^{\infty}(U;\pi)$. The elements of $\Gamma^{\infty}(U;\pi_k)$ having the latter property are called \emph{projectable}. This notion is motivated by the following simple observation which follows readily from~(\ref{pr1}).

\begin{Lemma}\label{pro}
Let $U\subset X$ be an open subset and $\Psi\in\Gamma^{\infty}(U;\pi_k)$. Then the map $p\mapsto \pi_{0,k}(\Psi(p))$ defines an element of $\Gamma^{\infty}(U;\pi)$, and $\Psi$ is projectable, if and only if one has $\jet^k(\pi_{0,k}\circ\Psi)=\Psi$.
 \end{Lemma}

\subsection{Partial dif\/ferential equations}\label{de}

The aim of this section is to give a precise global def\/inition of partial dif\/ferential equations and the solutions thereof in the setting of arbitrary f\/iber bundles. We shall f\/irst consider the general (possibly nonlinear) situation in Section~\ref{gene}. Then, in Section~\ref{lini}, we are going to relate everything with the corresponding classical linear concepts.

Throughout this section, let $\underline{\pi}\colon \underline{E}\to X$ be a second f\/iber bundle, with typical f\/iber $\underline{F}$ and f\/iber dimension~$\underline{n}$.

\subsubsection{General facts}\label{gene}

\begin{Definition}\label{dgl}\quad
\begin{enumerate}\itemsep=0pt
\item[a)] A subset $\IEE\subset \Jet^k(\pi)$ is called a~\emph{partial differential equation on~$\pi$ of order $\leq k$}, if ${\pi_k}_{|\IEE}\colon \IEE\to X$ is a f\/ibered submanifold of $\pi_k$.
\item[b)] Let $\IEE\subset \Jet^k(\pi)$ be a partial dif\/ferential equation on $\pi$ of order $\leq k$. Some $\psi\in\Gamma^{\infty} (p;\pi)$ is called a \emph{solution of $\IEE$ in $p$}, if $\jet^k_p\psi\in \IEE$. For an open $U\subset X$, a section $\psi\in\Gamma^{\infty}(U;\pi)$ will simply be called a \emph{solution} of $\IEE$, if $\psi$ is a solution of $\IEE$ in $p$ for every $p\in U$, that is, if $\mathrm{im}(\jet^k\psi)\subset \IEE$.
\end{enumerate}
\end{Definition}

\begin{Remark}\quad
\begin{enumerate}\itemsep=0pt
\item[a)] The def\/inition of a PDE we present here appears to go back to Gold\-schmidt~\cite{gold} and is now widely used in the geometric PDE literature \cite{lewis,pomm,pommPDEGT,reyes,sar,seiler}. One of its virtues is that it allows to clearly separate and globalize the notions ``partial dif\/ferential equation'', ``solution of a partial dif\/ferential equation'' and ``partial dif\/ferential operator''.
\item[b)] By the Cartan--Kuranishi prolongation theorem~\cite{kuranishi}, an analytic PDE either has no solutions or becomes involutive after f\/initely many prolongations, which means that it admits integral manifolds as stated in the Cartan--K\"ahler theorem~\cite{bryant}. In Goldschmidt's work, this observation is always in the background. Since we are mainly interested in the inf\/inite prolongation or in other words the formal solution space of a PDE as in Proposition and Def\/inition~\ref{defS}, Goldschmidt's approach f\/its perfectly for our purposes. For a slightly dif\/ferent general setup of non-linear PDEs of various orders (and their symbols) see~\cite{krug3}.
\item[c)] There exist non-f\/ibered submanifolds $E\subset \Jet^k(\pi)$ which one could refer to as ``partial dif\/ferential equations'' as well. The problem with this more general concept of a PDE though is that then independent variables might not really be independent anymore, cf.~\cite[Section~2.3]{seiler}. Moreover, only the f\/ibered variant leads in general to prof\/inite dimensional manifolds (cf.~Proposition~\ref{defS} below).
\end{enumerate}
\end{Remark}

\begin{Definition}\label{diffop} A morphism $h\colon \Jet^k(\pi)\to \underline{E}$ of f\/ibered manifolds over $X$ is called a \emph{partial differential operator of order $\leq k$ from $\pi$ to $\underline{\pi}$.}
\end{Definition}

Of course, the notion ``operator'' in Def\/inition \ref{diffop} is justif\/ied by the fact that as a morphism of f\/ibered manifolds, any $h$ as in Def\/inition~\ref{diffop} induces the morphism of set theoretic sheaves
\begin{gather*}
P^h:= h\circ \jet^k\colon \ \Gamma^{\infty}(\pi)\longrightarrow \Gamma^{\infty}(\underline{\pi}).
\end{gather*}
We def\/ine $\mathrm{D}^k(\pi,\underline{\pi})$ to be the set of all partial dif\/ferential operators of order $\leq k$ from $\pi$ to $\underline{\pi}$, and remark that the assignment $h\mapsto P^h$ induces an injection $P^{\bullet}$ of $\mathrm{D}^k(\pi,\underline{\pi})$ into the set theoretic sheaf morphisms $\Gamma^{\infty}(\pi)\to\Gamma^{\infty}(\underline{\pi})$. The connection between partial dif\/ferential operators and partial dif\/ferential equations is given in this abstract setting as follows: For every $h\in \mathrm{D}^k(\pi,\underline{\pi})$ and $O\in\Gamma^{\infty}(X;\underline{\pi})$, the set $\mathrm{Z}_{O}(h)$ is def\/ined by $\mathrm{Z}_{O}(h):=h^{-1}(\mathrm{im}(O))\subset \Jet^k(\pi)$, with the usual convention $\ker(h):=\mathrm{Z}_{O}(h)$ if $\pi$ and $\underline{\pi}$ are vector bundles and $h$ is linear. Observe that one has by def\/inition
\begin{gather*}
\mathrm{Z}_{O}(h)= \big\{a\in \Jet^k(\pi) \,|\, h(a)=O(\pi_k(a)) \big\}.
\end{gather*}

The following fact is well-known:

\begin{Proposition}\label{diffop2} If $h\in \mathrm{D}^k(\pi,\underline{\pi})$ has constant rank, and if $O\in\Gamma^{\infty}(X;\underline{\pi})$ fulfills $\mathrm{im}(O)\subset \mathrm{im}(h)$, then $\mathrm{Z}_{O}(h)\subset \Jet^k(\pi)$ is a partial differential equation.
\end{Proposition}

It is clear that with an open subset $U\subset X$, a section $\psi\in\Gamma^{\infty}(U;\pi)$ is a solution of $\mathrm{Z}_{O}(h)$ (in $p\in U$), if and only if one has $P^h_U(\psi)=O$ (in $p$).

Next, we explain how the af\/f\/ine structure of $\pi_{k-1,k}$ can be used to introduce the notion of ``operator symbols of (possibly nonlinear) partial dif\/ferential operators''. To avoid any confusion, we remark that with ``symbol'' we will exclusively mean ``principal symbol'' in this paper.

To this end, note that the assignment
\begin{gather*}
 \mu^{\pi}_k\colon \
 \pi^*_{k}\Sym^{k}\big( \pi_{\IT^*X} \big)\otimes \pi^*_{0,k}\Verti(\pi)\longrightarrow \Verti(\pi_k), \qquad
 {\mu^{\pi}_k}_{|a}(v):= \f{\Id} {\Id t} [a+tv]_{| t=0},
\end{gather*}
for $a\in \Jet^k(\pi)$, $v\in \Sym^{k}\big( \IT^*_{\pi_k(a)}X \big)\otimes \ker ( {\pi}_{|\pi_{0,k}(a)})$, is a (mono)morphism of vector bundles over~$\Jet^k(\pi)$. Note here that $\mu^{\pi}_k$ essentially extracts the pure $k$-th order part of vertical $k$-jets. Using the map~$\mu^{\pi}_k$, we can provide the following def\/inition (see also~\cite{bryant}):

\begin{Definition}\label{symbnon}
For every $h\in \mathrm{D}^k(\pi,\underline{\pi})$, the morphism $\sigma(h)$ of vector bundles over $h$ given by the composition
\begin{gather*}
 \xymatrix{\ar@{-->}[rr]^{\hspace{10mm}\sigma(h) }
 \pi^*_{k}\Sym^{k} \big( \pi_{\IT^*X} \big) \otimes \pi^*_{k,0}\Verti(\pi)\ar[dr]_{\mu^{\pi}_k} & &
 \Verti(\underline{\pi}) \\ & \Verti(\pi_k) \ar[ur]_{h_{\Verti}}
 }
 \end{gather*}
is called the \emph{operator symbol} of $h$.
\end{Definition}

Given a partial dif\/ferential operator, one can use its symbol to check whether it def\/ines a~partial dif\/ferential equation in the sense of Proposition~\ref{diffop2} (see also Theorem~\ref{mai} below):

\begin{Proposition}\label{susub} Let $h\in \mathrm{D}^k(\pi,\underline{\pi})$. If $\sigma(h)$ is surjective, then so is $h$. If $\sigma(h)$ is a submersion, then $h$ is a submersion, too, which in particularly means that for every $O\in\Gamma^{\infty}(X;\underline{\pi})$ with $\mathrm{im}(O)\subset \mathrm{im}(h)$ the set $\mathrm{Z}_{O}(h)\subset \Jet^k(\pi)$ is a partial differential equation.
\end{Proposition}

\begin{proof} We have the following commuting diagrams
\begin{gather*}
\begin{xy}
\xymatrix{
\Verti(\pi_k) \ar[rr]^-{h_{\Verti}}\ar[dd]_{(\pi_{k})^{\Verti}} && \Verti(\underline{\pi})\ar[dd]^{\underline{\pi}^{\Verti}}\\	&&\\
\Jet^k(\pi)\ar[rr]^{h} && \underline{E}
}
\end{xy} \qquad
\begin{xy}
\xymatrix{
\IT\Verti(\pi_k) \ar[rr]\ar[dd]& & \IT\Verti(\underline{\pi})\ar[dd]\\	&&\\
\IT\Jet^k(\pi)\ar[rr] && \IT\underline{E}
}
\end{xy}
\end{gather*}
where the maps for the second diagram are given by the tangential maps corresponding to the f\/irst one. If $\sigma(h)=h_{\Verti}\circ\mu^{\pi}_k$ is surjective, then so is $h_{\Verti}$ and $\underline{\pi}^{\Verti}\circ h_{\Verti}$, so that the f\/irst assertion follows from the f\/irst diagramm. If $\sigma(h)$ is a submersion, then one can use the analogous argument for the second diagram to deduce that~$\IT h$ has full rank everywhere.
\end{proof}

\subsubsection{Linear partial dif\/ferential equations}\label{lini}

We are now going to explain how the classical concepts of linear partial dif\/ferential equations and partial dif\/ferential operators f\/it into the general setting of Section~\ref{gene}. In fact, it will turn out that the notions ``linear partial dif\/ferential equation'' and ``linear partial dif\/ferential operator'' correspond to each other under natural assumptions in the following sense: $\ker (h)$ is a linear partial dif\/ferential
equation for every $k$-th order linear dif\/ferential operator and, conversely, for every linear $k$-th order partial dif\/ferential equation $E$ there exists a~(not uniquely determined) $k$-th order linear dif\/ferential operator $h$ with $E=\ker (h)$. Note that an analogous correspondence statement for the nonlinear case is only locally true~\cite{andy}. Finally, the space of linear partial dif\/ferential operators coincides with the space of classical linear partial dif\/ferential operators (see Theorem~\ref{peet} below). We begin with:

\begin{Definition}\label{lin}
Let $\pi$ and $\underline{\pi}$ be vector bundles.
\begin{enumerate}\itemsep=0pt
\item[a)] A subset $\IEE\subset \Jet^k(\pi)$ is called a~\emph{linear partial differential equation on $\pi$ of order $\leq k$}, if
 ${\pi_k}_{|\IEE}\colon \IEE\to X$ is a sub-vector-bundle of~$\pi_k$.
\item[b)] A morphism $h\colon \Jet^k(\pi)\to \underline{E}$ of vector bundles over $X$ is called a~\emph{linear partial differential operator of order $\leq k$ from $\pi$ to~$\underline{\pi}$}.
\end{enumerate}
\end{Definition}

\begin{Remark} Let $\pi$ and $\underline{\pi}$ be vector bundles. Then $h\in \mathrm{D}^k(\pi,\underline{\pi})$ is linear, if and only if $P^h_X\colon \Gamma^{\infty}(X;\pi)\to \Gamma^{\infty}(X;\underline{\pi})$ is linear. We denote the linear space of linear partial dif\/ferential operators by $\mathrm{D}^k_{\mathrm{lin}}(\pi,\underline{\pi})\subset \mathrm{D}^k(\pi,\underline{\pi})$ and remark that if $h\in\mathrm{D}^k_{\mathrm{lin}}(\pi,\underline{\pi})$, then one has $h^{(l)}\in \mathrm{D}^{k+l}_{\mathrm{lin}}(\pi,\underline{\pi}_l)$ for all $l\in\IN$.
\end{Remark}

Let us recall the def\/inition of ``classical'' linear partial dif\/ferential operators:

\begin{Definition}\label{peet0}
Let $\pi$ and $\underline{\pi}$ be vector bundles. A~\emph{classical linear partial differential operator of order $\leq k$} from $\pi$ to $\underline{\pi}$
is a morphism of sheaves
\begin{gather*}
D\colon \ \Gamma^{\infty}(\pi)\longrightarrow \Gamma^{\infty}(\underline{\pi})
\end{gather*}
with the following property: For every manifold chart $\tilde{x}\colon U\to \IR^m$ of $X$ for which there are frames $e_1,\dots,e_n\in\Gamma^{\infty}(U;\pi)$ and $\underline{e}_1,\dots,\underline{e}_{\underline{n}}\in\Gamma^{\infty}(U;\underline{\pi})$ there exist (necessarily unique) functions $D^{\alpha,\beta}_{I}\in\mathscr{C}^{\infty}(U)$ for $\alpha=1,\dots,n$, $\beta=1,\dots,\underline{n}$, and $I\in \IN^m_{0,k}$ such that one has for all $\psi^1,\dots,\psi^n\in\mathscr{C}^{\infty}(U)$
\begin{gather*}
D_U \left(\sum^n_{j=1}\psi^{\alpha}e_{\alpha}\right)= \sum^{\underline{n}}_{\beta=1}\sum^n_{\alpha=1}\sum_{I\in \IN^m_{0,k} } D^{\alpha,\beta}_{I} \f{\partial^{|I|} \psi^{\alpha}}{\partial \tilde{x}^I}\underline{e}_{\beta}.
\end{gather*}
The linear space of classical partial dif\/ferential operators will be denoted by $\mathrm{D}^k_{\mathrm{cl},\mathrm{lin}}(\pi,\underline{\pi})$.
\end{Definition}

Now one has:

\begin{Theorem}\label{peet}
Let $\pi$ and $\underline{\pi}$ be vector bundles.
\begin{enumerate}\itemsep=0pt
\item[$a)$] $P^{\bullet}$ induces the isomorphism of linear spaces
 \begin{gather*}
 P^{\bullet}_{\mathrm{lin}}\colon \ \mathrm{D}^k_{\mathrm{lin}}(\pi,\underline{\pi})\longrightarrow \mathrm{D}^k_{\mathrm{cl},\mathrm{lin}}(\pi,\underline{\pi}),\qquad h\longmapsto P^h.
 \end{gather*}

\item[$b)$] If $h\in\mathrm{D}^k_{\mathrm{lin}}(\pi,\underline{\pi})$ has constant rank, then $\ker(h)\subset\Jet(\pi_k)$ is a linear partial differential equation. Conversely, if $\IEE\subset \Jet^k(\pi)$ is a linear partial differential equation, then there is a~vector bundle $\underline{\underline{\pi}}\colon \underline{\underline{E}}\to X$ and an $\underline{\underline{h}}\in \mathrm{D}^k_{\mathrm{lin}}\left(\pi,\underline{\underline{\pi}}\right)$ with constant rank such that $\IEE= \ker\left(\underline{\underline{h}}\right)$.
\end{enumerate}
\end{Theorem}

\begin{proof} a) We f\/irst have to show that $P^{\bullet}_{\mathrm{lin}}$ is well-def\/ined, which means that for any $h\in\mathrm{D}^k_{\mathrm{lin}}(\pi,\underline{\pi})$, $P^{h}$ is in $\mathrm{D}^k_{\mathrm{cl},\mathrm{lin}}(\pi,\underline{\pi})$. It is then clear that $P^{\bullet}_{\mathrm{lin}}$ is a linear monomorphism. To this end, let $\tilde{x}$, $e_{\alpha}$, $\underline{e}_{\beta}$, $\psi^{\alpha}$ be as in Def\/inition \ref{peet0} and let $a_{\alpha}$ be a basis for $F$. Then we have the vector bundle chart
\begin{gather*}
\phi\colon \ \pi^{-1}(U)\to U \times F, \qquad \sum^n_{\alpha=1} v^{\alpha}e_{\alpha}(p)\longmapsto \left(p,\sum^n_{\alpha=1} v^{\alpha} a_{\alpha}\right),\qquad p\in U,
\end{gather*}
so that we get the f\/ibered chart
\begin{gather*}
(x,u)\colon \ \pi^{-1}(U)\longrightarrow \IR^m\times \IR^n,\qquad \sum^n_{\alpha=1} v^{\alpha}e_{\alpha}(p)\longmapsto \big(\tilde{x}(\pi(p)),\big(v^1,\dots,v^n\big)\big)
\end{gather*}
of $\pi$ as in Remark~\ref{ReCha}(c). Then
\begin{gather*}
(\pi_k,u_k) \colon \ \pi_{k}^{-1}(U)\longrightarrow U\times \jetfib(m,0,k)^n
\end{gather*}
is a vector bundle chart of $\pi_k$ by Lemma \ref{struc}, and we get the frame $e_{I,\alpha}\in\Gamma(U;\pi_k)$, $\alpha=1,\dots, n$, $I\in \IN^m_{0,k}$, given by
\begin{gather*}
e_{I,\alpha}:= (\pi_k,u_k )^{-1}(\bullet,\delta_{I,\alpha}).
\end{gather*}
Hereby, $\delta_{I,\alpha}\colon \IN^m_{0,k}\to \IR^n$ is def\/ined by $\delta_{I,\alpha}(J):=1_{\alpha}$, if $I=J$, and to be $0$ elsewhere. Since $h$ is a~homomorphism of linear bundles over $X$, there are uniquely determined $h^{\alpha,\beta}_{I}\in\mathscr{C}^{\infty}(U)$ such that one has for all
$\alpha=1,\dots,n$, $I\in \IN^m_{0,k}$ and $\psi^{\alpha,I}\in \mathscr{C}^{\infty}(U)$
\begin{gather*}
h\left(\sum^n_{\alpha=1}\sum_{I\in \IN^m_{0,k} }\psi^{\alpha,I}e_{\alpha,I}\right) =\sum^{\underline{n}}_{\beta=1}\sum^n_{\alpha=1}\sum_{I\in \IN^m_{0,k} } h^{\alpha,\beta}_{I}\psi^{\alpha,I}\underline{e}_{\beta}.
\end{gather*}
The proof of the asserted well-def\/inedness of $P^{\bullet}_{\mathrm{lin}}$ is completed by observing that by the above construction of the frame $e_{I,\alpha}$ for $\Gamma^{\infty}(U;\pi_k)$ the following equality holds true:
\begin{gather*}
\jet^k \left(\sum^n_{\alpha=1}\psi^{\alpha}e_{\alpha}\right) = \sum^n_{\alpha=1}\sum_{I\in \IN^m_{0,k} }\f{\partial^{|I|} \psi^{\alpha}}{\partial \tilde{x}^I}e_{\alpha,I}.
\end{gather*}

In order to prove surjectivity of $P^{\bullet}_{\mathrm{lin}}$, let $D\in \mathrm{D}^k_{\mathrm{cl},\mathrm{lin}}(\pi,\underline{\pi})$, and let
$\jet^k_p\psi\in \Jet(\pi_k)$, with $p$ from an open subset $U\subset X$. Then $h(\jet^k_p\psi):= D_U\psi(p)$ gives rise to a well-def\/ined element $h\in\mathrm{D}^k_{\mathrm{lin}}(\pi,\underline{\pi})$, which of course satisf\/ies $P^h_{\mathrm{lin}} = D$.

b) The f\/irst fact is well-known. For the second assertion, we can simply take $\underline{\underline{E}}\to X$ to be given by the quotient bundle $\Jet^k(\pi)/\IEE\to X$, and $\underline{\underline{h}}$ to be given by the canonical projection $\Jet^k(\pi)\to \Jet^k(\pi)/\IEE$ (see \cite[Proposition~3.10]{lewis} for a more general statement).
\end{proof}

Finally, we explain in which sense the classical concept of linear operator symbols f\/its into the general setting of Section~\ref{gene}.

Let $\pi$ and $\underline{\pi}$ be vector bundles for the moment. From the canonical identif\/ication of $\operatorname{ker}(\IT\pi_{| e})$ with $\pi^{-1}(\pi(e))$, for $e\in E$, (and analogous ones for $\underline{\pi}$), we obtain canonical morphisms of vector bundles over the base map~$\pi$ (resp.~$\underline{\pi}$)
\begin{gather}\label{iso1}
\begin{split}
& \sigma^{\pi}\colon \ \Verti(\pi) \longrightarrow E, \\
&{\sigma^{\pi}}_{|\ker({\IT\pi}_{|e})}\colon \ \ker({\IT\pi}_{|e}) \longrightarrow \pi^{-1}(\pi(e)),\qquad e\in E,\qquad \text{and} \\
\end{split}
\\ \label{iso2}
\begin{split}
& \sigma^{\underline{\pi}}\colon \ \Verti(\underline{\pi}) \longrightarrow \underline{E}, \\
&{\sigma^{\underline{\pi}}}_{|\ker(\IT\underline{\pi}_{|\underline{e}})}\colon \ \ker(\IT\underline{\pi}_{|\underline{e}}) \longrightarrow \underline{\pi}^{-1}(\underline{\pi}(\underline{e})),\qquad \underline{e}\in \underline{E},
\end{split}
\end{gather}
 which both are f\/iberwise isomorphisms.
 It follows that for each $k\in \IN^*$ the map
\begin{gather}
\begin{split}\label{iso3}
\sigma^{\pi}_k\colon \ & \pi^*_k\Sym^{k}\big( \pi_{\IT^*X} \big)\otimes \pi^*_{k,0}\Verti(\pi) \longrightarrow \Sym^{k} \big( \pi_{\IT^*X} \big)\otimes E, \\
& {\sigma^{\pi}_k} (v\otimes w):=v\otimes \sigma^{\pi}(w),
\end{split}
\end{gather}
where \looseness=-1 $v\otimes w\in\operatorname{Sym}^{k} \big( \IT^*_{\pi_k(a)}X\big)\otimes \operatorname{ker}\big( {\IT\pi}_{|\pi_{0,k}(a)} \big)$ for $a\in\Jet^k(\pi)$, is a morphism of vector bundles over the base map $\pi_k$ and also acts by ismorphisms, f\/iberwise. Furthermore, for later reference, we record that there is a~canonical (mono)morphism of vector bundles over $X$ which is def\/ined by
\begin{align*}
\begin{split}
 \mu^{\pi}_{k,\mathrm{lin}} \colon \ & \Sym^{k} \big( \pi_{\IT^*X} \big) \otimes E\longrightarrow \Jet^k(\pi), \\
&\Id f_1(p)\odot\cdots \odot\Id f_k(p) \otimes \psi(p)\longmapsto \jet^k_p(f_1\cdots f_k\psi).
\end{split}
\end{align*}
Here, the $f_j$ run through the elements of $\mathscr{C}^{\infty}(p;X)$ satisfying $f_j(p)=0$, and $\psi\in\Gamma^{\infty}(p;\pi)$. Ana\-lo\-gously to $\mu^{\pi}_k$, the map $\mu^{\pi}_{k,\mathrm{lin}}$ also extracts the pure $k$-th order part of $k$-jets in an appropriate sense (taking into account the canonical isomorphisms~(\ref{iso2}) and~(\ref{iso3})).

The following result recalls the classical def\/inition of linear operator symbols and shows the naturality of Def\/inition~\ref{symbnon}, in the sense that in the linear case, the linear operator symbol coincides with the operator symbol up to the canonical isomorphisms~(\ref{iso2}) and~(\ref{iso3}):

\begin{Propanddef}\label{PropLinSymb} Let $\pi$ and $\underline{\pi}$ be vector bundles and let $h\in\mathrm{D}^k_{\mathrm{lin}}(\pi,\underline{\pi})$.
\begin{enumerate}\itemsep=0pt
\item[$a)$] There is a unique morphism of vector bundles over $X$
\begin{gather*}
\sigma_{\mathrm{lin}}(h)\colon \ \Sym^{k}(\pi_{\IT^*X})\otimes E\longrightarrow \underline{E}
\end{gather*}
with the following property: For every manifold chart $\tilde{x}\colon U\to \IR^m$ of $X$ for which there are frames $e_1,\dots,e_n\in\Gamma^{\infty}(U;\pi)$ and $\underline{e}_1,\dots,\underline{e}_{\underline{n}}\in\Gamma^{\infty}(U;\underline{\pi})$, one has
\begin{gather}
\sigma_{\mathrm{lin}}(h)\left( \sum_{I\in\IN^m_{k,k}}\sum^n_{\alpha=1}v^{\alpha}_I\Id \tilde{x}_{I}\otimes e_{\alpha}\right)=\sum^{\underline{n}}_{\beta=1}\sum^{n}_{\alpha=1}\sum_{I\in\IN^m_{k,k}}
(P^h_{\mathrm{lin}})^{\alpha,\beta}_{I}v^{\alpha}_I\underline{e}_{\beta},\label{local}
\end{gather}
where $v^{\alpha}_I\in\mathscr{C}^{\infty}(U)$, and where we have used the notation from Definition~{\rm \ref{peet0}} and Theorem~{\rm \ref{peet}}. The morphism $\sigma_{\mathrm{lin}}(h)$ is called the \emph{linear operator symbol of~$h$}.

\item[$b)$] The following diagram commutes
\begin{gather*}
\begin{xy}
\xymatrix{
 \pi^*_{k}\Sym^{k}(\pi_{\IT^*X})\otimes \pi^*_{k,0}\Verti(\pi)\ar[rr]^-{\sigma(h)} \ar[dd]_-{\sigma^{\pi}_k} &&\Verti(\underline{\pi})
 \ar[dd]^{\sigma^{\underline{\pi}}} \\	&&\\
\Sym^{k}(\pi_{\IT^*X})\otimes E\ar[rr]^{\hspace{10mm}\sigma_{\mathrm{lin}}(h)} && \underline{E}
}
\end{xy}
\end{gather*}
\end{enumerate}
\end{Propanddef}

\begin{proof} a) Here, one only has to prove that the representation~(\ref{local}) does not depend on a~particular choice of local data. In fact, the easiest way to see this, is to note that one can simply def\/ine $\sigma_{\mathrm{lin}}(h)$ by the diagram
\begin{gather*}
 \xymatrix{\ar@{-->}[rr]^{\hspace{10mm}\sigma_{\mathrm{lin}}(h) } \Sym^{k}(\pi_{\IT^*X})\otimes E\ar[dr]_{\mu^{\pi}_{k,\mathrm{lin}}} & & \underline{E} \\
 & \Jet^k(\pi) \ar[ur]_{h}
 }
 \end{gather*}
To see that $\sigma_{\mathrm{lin}}(h)$ def\/ined like this satisf\/ies (\ref{local}), let $\tilde{x}\colon U\to \IR^m$ be a manifold chart of $X$ such that there are frames $e_1,\dots,e_n\in\Gamma^{\infty}(U;\pi)$, $\underline{e}_1,\dots,\underline{e}_{\underline{n}}\in\Gamma^{\infty}(U;\underline{\pi})$. Then, as in the proof of Theorem~\ref{peet}(a), picking a basis $a_{\alpha}$ for $F$, we get the corresponding frame $e_{I,\alpha}\in\Gamma^{\infty}(U;\pi_k)$, $\alpha=1,\dots n$, $I\in \IN^m_{0,k}$, and we denote the representation of $h$ with respect to $e_{I,\alpha}$ and $\underline{e}_{\beta}$ by $h^{\alpha,\beta}_I$. Furthermore, by the proof of Theorem~\ref{peet}(a), we have $h^{\alpha,\beta}_I=(P^h_{\mathrm{lin}})^{\alpha,\beta}_{I}$. Now one has
\begin{gather*}
\mu^{\pi}_{k,\mathrm{lin}}\left(\sum_{I\in\IN^m_{k,k}}\sum^n_{\alpha=1}v^{\alpha}_I\Id \tilde{x}_{I}\otimes e_{\alpha}
\right)= \sum_{I\in\IN^m_{k,k}}\sum^n_{\alpha=1}v^{\alpha}_I e_{I,\alpha},
\end{gather*}
so that
\begin{gather*}
h\circ\mu^{\pi}_{k,\mathrm{lin}}\left( \sum_{I\in\IN^m_{k,k}}\sum^n_{\alpha=1}v^{\alpha}_I\Id \tilde{x}_{I}\otimes e_{\alpha}\right)=\sum^{\underline{n}}_{\beta=1}\sum^{n}_{\alpha=1}\sum_{I\in\IN^m_{k,k}}h^{\alpha,\beta}_Iv^{\alpha}_I\underline{e}_{\beta},
\end{gather*}
which completes the proof of part~a).

b) Let $a\in \Jet^k(\pi)$ be arbitrary, and let $\tilde{x}\colon U\to \IR^m$ be a manifold chart of $X$ around $\pi_k(a)$ such that there are frames $e_1,\dots,e_n\in\Gamma^{\infty}(U;\pi)$, $\underline{e}_1,\dots,\underline{e}_{\underline{n}}\in\Gamma^{\infty}(U;\underline{\pi})$. Then, again as in the proof of Theorem~\ref{peet}(a), picking a basis $a_{\alpha}$ for $F$ and a basis $\underline{a}_{\beta}$ for $\underline{F}$, we get the corresponding adapted coordinates
\begin{gather*}
(x,u)\colon \ \pi^{-1}(U)\longrightarrow \IR^m\times \IR^n,\qquad (\underline{x},\underline{u})\colon \ \underline{\pi}^{-1}(U)\longrightarrow \IR^m\times \IR^{\underline{n}}.
\end{gather*}
We can expand an arbitrary
\begin{gather*}
v\in\Sym^{k}\big(\IT^*_{\pi_k(a)}X\big)\otimes \ker(\mathrm{T\pi}_{| \pi_{0,k}(a)})
\end{gather*}
uniquely as
\begin{gather*}
v=\sum_{I\in\IN^m_{k,k}}\sum^n_{\alpha=1}v^{\alpha}_I \Id \tilde{x}_{I \,|\, \pi_k(a)}\otimes \f{\partial}{\partial u^{\alpha}}_{| \pi_{0,k}(a)},\qquad v^{\alpha}_I\in\IR,
\end{gather*}
so that
\begin{gather*}
\sigma^{\pi_k}(v)=\sum_{I\in\IN^m_{k,k}}\sum^n_{\alpha=1}v^{\alpha}_I \Id \tilde{x}_{I}\otimes e_{\alpha \,|\, \pi_k(a)},
\end{gather*}
and we arrive at
\begin{gather}
\sigma_{\mathrm{lin}}(h)\circ\sigma^{\pi_k}_{| v} =\sum^{\underline{n}}_{\beta=1}\sum^{n}_{\alpha=1}\sum_{I\in\IN^m_{k,k}}(P^h_{\mathrm{lin}})^{\alpha,\beta}_{I}v^{\alpha}_I\underline{e}_{\beta \,|\, \pi_k(a)}.\label{hil}
\end{gather}
Let us now evaluate $\sigma^{\underline{\pi}}\circ\sigma(h)=\sigma^{\underline{\pi}}\circ h_{\Verti}\circ \mu^{\pi}_k$ in~$v$: By Proposition~\ref{jetmf} and Lemma~\ref{struc} we have the frame
\begin{gather*}
\f{\partial}{\partial u^{k}_{I,\alpha}}\in\Gamma^{\infty}\big(\pi^{-1}_k(U); (\pi_k)^{\Verti} \big),\qquad I\in\IN^m_{0,k},\qquad \alpha=1,\dots,n .
\end{gather*}
As $h$ is a linear morphism, the linear morphism $h_{\Verti}$ is represented with respect to the frames $\f{\partial}{\partial u^{\alpha}_{k,I}}$ and $\f{\partial}{\partial \underline{u}^{\beta}}$ precisely by the functions
\begin{gather*}
h^{\alpha,\beta}_I\circ\big(\pi_{k \,|\, \pi^{-1}_k(U)}\big)\in\mathscr{C}^{\infty}\big(\pi^{-1}_k(U)\big),
\end{gather*}
where $h^{\alpha,\beta}_I\in\mathscr{C}^{\infty}(U)$ is the representation of $h$ with respect to $(x,u)$ and $(\underline{x},\underline{u})$ (cf.\ the proof of Theorem~\ref{peet}(a)). Thus, in view of
\begin{gather*}
\mu^{\pi}_k(v)=\sum^{n}_{\alpha=1}\sum_{I\in\IN^m_{k,k}}v^{\alpha}_I \f{\partial}{\partial u^{\alpha}_{k,I} }_{| a},
\end{gather*}
we have
\begin{gather*}
\sigma(h)_{|_v}=\sum^{\underline{n}}_{\beta=1}\sum^{n}_{\alpha=1}\sum_{I\in\IN^m_{k,k}} v^{\alpha}_I h^{\alpha,\beta}_{I \,|\, \pi_k(a)} \f{\partial}{\partial \underline{u}^{\beta} }_{| \pi_{0,k}(a)},
\end{gather*}
so that
\begin{gather*}
\sigma^{\underline{\pi}}\circ\sigma(h)_{| v}= \sum^{\underline{n}}_{\beta=1}\sum^{n}_{\alpha=1}\sum_{I\in\IN^m_{k,k}} h^{\alpha,\beta}_Iv^{\alpha}_I \underline{e}_{\beta \,|\, \pi_k(a)} .
\end{gather*}
But this is equal to~(\ref{hil}), in view of $(P^h_{\mathrm{lin}})^{\alpha,\beta}_{I}=h^{\alpha,\beta}_I$. The claim follows.
\end{proof}

\subsection[The manifold of $\infty$-jets and formally integrable PDEs]{The manifold of $\boldsymbol{\infty}$-jets and formally integrable PDEs}\label{inf}

Throughout Section~\ref{inf}, $\pi$ will again be an arbitrary f\/iber bundle.

Finally, in this section we are going to make contact with the abstract theory on prof\/inite dimensional manifolds from Section~\ref{pfd}: We are going to prove that the space of ``$\infty$-jets'' in $\pi$ canonically becomes a prof\/inite dimensional manifold (see Proposition~\ref{borel}), and that the space of ``formal solutions'' of a ``formally integrable'' partial dif\/ferential equation on~$\pi$ canonically is a~prof\/inite dimensional submanifold of the latter (see Proposition~\ref{defS}).

We start by introducing the space of $\infty$-jets. In analogy to Def\/inition~\ref{equiv}, we have:

\begin{Definition} Let $p\in X$. Any two $\psi, \varphi\in\Gamma^{\infty} (p;\pi)$ are called \emph{$\infty$-equivalent at~$p$}, if $\psi(p)=\varphi(p)$ and if for every f\/ibered chart $(x,u)\colon W\to\IR^m\times\IR^n$ of $\pi$ with $W\cap \pi^{-1}(p)\ne \varnothing$ one has
\begin{gather*}
\f{\partial^{|I|} (u^{\alpha}\circ\psi)}{\partial \tilde{x}^I} (p)=\f{\partial^{|I|} (u^{\alpha}\circ\varphi)}{\partial \tilde{x}^I} (p)
\end{gather*}
for all $\alpha=1,\dots,n$ and all $I\in \IN^m$ with $1\leq |I|< \infty$. The corresponding equivalence class $\jet^{\infty}_p \psi$ of $\psi$ is called the \emph{$\infty$-jet of~$\psi$ at~$p$}.
\end{Definition}

\begin{Remark}
In view of Remark \ref{fadibru}, $\infty$-equivalence also only has to be checked in \emph{some} f\/ibered chart.
\end{Remark}

It will be convenient in what follows to set $\Jet^{-1}(\pi):=X$, $\jet^{-1}_p\psi_p := p$, and $\pi_{-1,0}:=\pi$. We def\/ine
\begin{gather*}
\Jet^{\infty}(\pi):= \bigcup_{p\in X} \big\{ \jet^{\infty}_p\psi \,|\, \psi \in\Gamma^{\infty} (p;\pi) \big\},
\end{gather*}
and obtain for every $i\in\IZ_{\geq -1}$ a surjective map
\begin{align}
 \pi_{i,\infty}\colon \ \Jet^{\infty}(\pi)\longrightarrow \Jet^{i}(\pi),\qquad \jet^{\infty}_p\psi\longmapsto \jet^{i}_p\psi .
 \label{initial}
\end{align}
We equip $\Jet^{\infty}(\pi)$ with the initial topology with respect to the maps $\pi_{i,\infty}$, $i\in\IZ_{\geq -1}$. Furthermore, we def\/ine $\mathscr{C}^{\infty}_{\pi}$ to be the sheaf on $\Jet^{\infty}(\pi)$, whose section space $\mathscr{C}^{\infty}_{\pi}(U)$ over an open $U\subset \Jet^{\infty}(\pi)$ is given by the set of all $f\in \mathscr{C}(U)$ such that for every $x\in U$ there is an $i\in \IZ_{\geq -1}$, an open $U_i\subset \Jet^{i}(\pi)$ and an $f_i\in \mathscr{C}^{\infty}(U_i)$ with $x\in \pi_{i,\infty}^{-1}(U_i) \subset U$ and
\begin{gather*}
 f_{| \pi_{i,\infty}^{-1}(U_i)}=f_i\circ {\pi_{i,\infty}}_{|\pi_{i,\infty}^{-1}(U_i)} .
\end{gather*}
In particular, $(\Jet^{\infty}(\pi),\mathscr{C}^{\infty}_{\pi})$ becomes a locally $\IR$-ringed space. Now observe that we have, in view of~(\ref{proj}), a~smooth projective system $\big( \Jet^i(\pi), \pi_{i,j}\big)$, which graphically can be depicted by
\begin{gather}\label{lim}
\Jet^{-1}(\pi)\xleftarrow{\pi_{-1,0}} \Jet^0(\pi) \xleftarrow{\pi_{0,1}}
\cdots \longleftarrow \Jet^{i}(\pi)\xleftarrow{\pi_{i,i+1}} \Jet^{i+1}(\pi) \longleftarrow \cdots,
\end{gather}
together with a family of continuous maps
\begin{gather*}
 \pi_{i,\infty}\colon \ \Jet^{\infty}(\pi)\longrightarrow \Jet^{i}(\pi), \qquad i\in\IZ_{\geq -1}.
\end{gather*}
These data have the following crucial property.
\begin{Propanddef}\label{borel}
The family $\big( \Jet^i(\pi), \pi_{i,j} , \pi_{i,\infty} \big)$ is a smooth projective representation of $(\Jet^{\infty}(\pi),\mathscr{C}^{\infty}_{\pi})$. In particular, when equipped with the corresponding pfd structure, $(\Jet^{\infty}(\pi),\mathscr{C}^{\infty}_{\pi})$ canonically becomes a smooth profinite
dimensional manifold, called the \emph{manifold of $\infty$-jets given by $\pi$.}
\end{Propanddef}

\begin{proof} Let $\Jet^{\infty}(\pi)':=\lim \limits_{\longleftarrow \atop i\in \IZ_{\geq -1} }\Jet^{i}(\pi)$ denote the canonical projective limit of~(\ref{lim}), that means let
\begin{gather*}
 \Jet^{\infty}(\pi)' = \bigg\{ b=(b_{-1},b_0,b_1,\dots)\in\prod_{i\in \IZ_{\geq -1}}\Jet^{i}(\pi)\,|\, b_i=\pi_{i,j}(b_j)
 \text{ for all $i\leq j$}\bigg\} ,
\end{gather*}
and let $\pi_{i,\infty}'\colon \Jet^{\infty}(\pi)'\to \Jet^{i}(\pi)$ denote the canonical projections. We are going to prove the existence of a homeomorphism $\Xi$ such that the diagrams
\begin{gather} \label{dgq}
\begin{split}&
 \xymatrix{
\ar@<2pt>[rr]^{\Xi } \Jet^{\infty}(\pi)\ar[dr]_{\pi_{i,\infty}} & & \ar@<2pt>[ll]^{\Xi^{-1} }\Jet^{\infty}(\pi)'\ar[dl]^{\pi_{i,\infty}'} \\
 & \Jet^i(\pi)
 }
\end{split}
 \end{gather}
commute for all $i\in \IZ_{\geq -1}$. Then the universal property of $(\Jet^{\infty}(\pi)',\pi_{i,\infty}')$ will directly imply the same property for $(\Jet^{\infty}(\pi),\pi_{i,\infty})$, which is precisely (PFM1). As a consequence, (PFM2) is trivially satisf\/ied by the def\/inition of the structure sheaf $\mathscr{C}^{\infty}_{\pi}$.

We now simply def\/ine $\Xi(a)_j:=\pi_{j,\infty}(a)$ for $a\in\Jet^{\infty}(\pi)$ and $ j\in \IZ_{\geq -1}$. Then it is obvious that~$\Xi$ is a well-def\/ined injective map, and that the $\Xi$-diagram in~(\ref{dgq}) commutes. In particular, the continuity of $\Xi$ is directly implied by that of the maps $\pi_{i,\infty}$. In order to see that $\Xi$ is surjective and that $\Xi^{-1}$ is continuous, let us recall that Borel's theorem states
that for any map
\begin{gather*}
t\colon \ \IN^m=\bigcup_{j\in\IN}\IN^m_{0,j}\longrightarrow\IR^n
\end{gather*}
there is a smooth function $\tilde{\psi}\colon \IR^m\to \IR^n$ such that $t_I=\partial_I\tilde{\psi}(0)/I!$ for all $I\in\IN^m$. Let $b\in\Jet^{\infty}(\pi)'$ and let $\tilde{x}\colon U\to\IR^m$ be a manifold chart of $X$ around $b_{-1}$ with $\tilde{x}(b_{-1})=0$. Choosing furthermore a bundle chart $\phi\colon \pi^{-1}(U)\to U\times F$ and a manifold chart $\tilde{u}\colon B\to \IR^n$ of $F$, we get the f\/ibered chart
\begin{gather*}
(x,u):=(\tilde{x}\circ \pi,\tilde{u}\circ\mathrm{pr}_2\circ \phi)\colon \ \pi^{-1}(U)\longrightarrow \IR^m\times \IR^n
\end{gather*}
of $\pi$ by Remark~\ref{ReCha}(c). Moreover, the function $t\colon \IN^m\to\IR^n$, $t_I:= u_{j,I}(b_j)/ I!$, if $I\in\IN^m_{0,j}$, is well-def\/ined. Borel's theorem then produces a function $\tilde{\psi}\colon \IR^m\to \IR^n$ such that $t_I=\partial_I\tilde{\psi}(0)/I!$. It is clear that the section $\psi\in\Gamma^{\infty}_{b_{-1}}(\pi)$ def\/ined by
\begin{gather*}
\psi:=\phi^{-1}\big(\bullet,\tilde{u}^{-1}\circ\tilde{\psi}\circ\tilde{x}\big)
\end{gather*}
satisf\/ies $x_j(\jet^j_{b_{-1}}\psi)=0$, and $u_{j,I}(\jet^j_{b_{-1}}\psi)=u_{j,I}(b_j)$ for all $j \!\in\! \IZ_{\geq -1}$ and $I\in\IN^m_{0,j}$, thus \mbox{$\Xi(\jet^{\infty}_{b_{-1}}\psi)=b$}, and $\Xi$ is surjective, indeed. Furthermore, by the construction of~$\Xi^{-1}(b)$, it is also clear that the $\Xi^{-1}$-diagram in (\ref{dgq}) commutes, so that the continuity of $\Xi^{-1}$ trivially follows from that of the~$\pi_{i,\infty}'$. This completes the proof.
\end{proof}

\begin{Remark} An important additional structure on the prof\/inite dimensional manifold $\Jet^{\infty}(\pi)$ is given by its \emph{Cartan distribution}, a canonically given (involutive) distribution. This additional structure is the key to the so-called \emph{secondary calculus}. We shall not touch this here and refer the interested reader to~\cite{vino}.
\end{Remark}

Next, we will prepare the notion of formal integrability. Let us f\/irst note the following simple result:

\begin{Propanddef}\label{ope}\quad
\begin{enumerate}\itemsep=0pt
\item[$a)$] For every $h\in \mathrm{D}^k(\pi,\underline{\pi})$ there exists a unique $h^{(l)}\in \mathrm{D}^{k+l}(\pi,\underline{\pi}_l)$ such that the following diagram of set theoretic sheaf morphisms
\begin{gather*}
\begin{xy}
\xymatrix{
\Gamma^{\infty}(\pi_{k+l}) \ar@{-->}[rr]^-{h^{(l)}} && \Gamma^{\infty}(\underline{\pi}_{l})\\	&&\\
\Gamma^{\infty}(\pi)\ar[uu]^{\jet^{k+l}}\ar[rr]^{h\circ \jet^{k}} && \Gamma^{\infty}(\underline{\pi})\ar[uu]_{\jet^{l}}
}
\end{xy}
\end{gather*}
commutes. The partial differential operator $h^{(l)}$ is called the \emph{$l$-jet prolongation of~$h$}.

\item[$b)$] The partial differential operator
\begin{gather*}
\iota^{\pi}_{l,k} \big({=}\mathrm{id}_{\Jet^k(\pi)}^{(l)}\big) \colon \ \Jet^{k+l}(\pi) \longrightarrow \Jet^{l}(\pi_k),\qquad \jet^{k+l}_p\psi\longmapsto \jet^{l}_p(\jet^k\psi)
\end{gather*}
is an embedding of manifolds.
\end{enumerate}
\end{Propanddef}

Let $\IEE\subset \Jet^k(\pi)$ be an arbitrary partial dif\/ferential equation for the moment. Since, by def\/inition, the map
\begin{gather*}
 {\pi_k}_{|\IEE}\colon \ \Jet^k(\pi)\supset \IEE\longrightarrow X
\end{gather*}
is again a f\/ibered manifold, there exists for every $l\in\IN$ an obvious well-def\/ined map
\begin{gather*}
 \iota_{l,\IEE}\colon \ \Jet^l({\pi_k}_{|\IEE})\longrightarrow \Jet^l(\pi_k),
\end{gather*}
which comes from considering a locally def\/ined section in ${\pi_k}_{|\IEE}$ as taking values in $\Jet^k(\pi)$.

\begin{Definition}\label{regu}
 Let $\IEE\subset \Jet^k(\pi)$ be a partial dif\/ferential equation. Then the set
\begin{gather*} 
\IEE^{(l)}:=
\begin{cases}
 \IEE, & \text{for $l=0$},\\
 \iota_{l,k}^ {\pi,-1}\big( \iota_{l,\IEE}\big(\Jet^l(\pi_{k \,|\, \IEE})\big)\big) \subset \Jet^{k+l}(\pi) ,
 & \text{for $l\in \IN^*$},
\end{cases}
\end{gather*}
is called the \emph{$l$-jet prolongation} of $\IEE$.
\end{Definition}

If the underlying partial dif\/ferential equation is actually given by a partial dif\/ferential operator, then there is an explicit description of the corresponding $l$-jet prolongation~\cite[p.~294]{gold}:

\begin{Proposition}\label{oper}
Let $h\in\mathrm{D}^k(\pi,\underline{\pi})$ with constant rank and let $O\in\Gamma^{\infty}(X;\underline{\pi})$ with $\mathrm{im}(O)\subset \mathrm{im}(h)$. Then one has, for every $l\in\IN$,
\begin{gather*}
\mathrm{Z}_{O}(h)^{(l)}= \mathrm{Z}_{\jet^l O}\big( h^{(l)}\big)\subset\Jet^{k+l}(\pi).
\end{gather*}
\end{Proposition}

Let us note the simple fact that the following diagramm commutes, for every $r\in\IN$,
\begin{gather*}
\begin{xy}
\xymatrix{
\Jet^{k+l+r}(\pi)\ar[rr]^-{\iota^{\pi}_{l+r,k}} \ar[dd]_-{\pi_{k+l,k+l+r}} && \Jet^{l+r}(\pi_k)\ar[dd]^-{(\pi_k)_{l+r,l}} \\	&&\\
\Jet^{k+l}(\pi)\ar[rr]^{\iota^{\pi}_{l,k}} && \Jet^l(\pi_k)
}
\end{xy}
\end{gather*}
Applying this in the case $r=1$ implies $\pi_{k+l,k+l+1}\big(\IEE^{(l+1)}\big)\subset \IEE^{(l)}$ for any partial dif\/ferential equation $\IEE\subset \Jet^k(\pi)$ and every $l\in\IN$, so that we obtain the maps
\begin{gather}
 \IEE^{(l+1)}\longrightarrow \IEE^{(l)},\qquad a\longmapsto \pi_{k+l,k+l+1}(a). \label{for}
\end{gather}
Now we have the tools to give
\begin{Definition}\label{formi}
A partial dif\/ferential equation $\IEE\subset \Jet^k(\pi)$ is called \emph{formally integrable}, if $\IEE^{(l)}$ is a submanifold of $\Jet^{k+l}(\pi)$ and if~(\ref{for}) is a f\/ibered manifold for every $l\in\IN$.
\end{Definition}

\begin{Remark}\quad
\begin{enumerate}\itemsep=0pt
\item[a)] Here, it should be noted that $E$ itself can always be considered as a trivial formally integrable partial dif\/ferential equation on $\pi$ of order $0$, where in this case one has $E^{(l)}=\Jet^l(\pi)$ for all $l\in\IN$.

\item[b)] Furthermore, there are abstract cohomological tests for partial dif\/ferential equations to be formally integrable~\cite{gold}. In fact, we will use such a test in the proof of Theorem~\ref{mai} below; we refer the reader to~\cite{pomm} and particularly to~\cite{seiler} for the algorithmic aspects of these tests. Although it can become very involved to verify these test properties in particular examples, formal integrability could be shown for some particular equations of mathematical physics such as the Yang--Mills--Higgs equations \cite{gia2,gia} or the Einstein's f\/ield equations~\cite{krug1}. In accordance with this, Theorem \ref{mai} below states that all reasonable (possibly nonlinear) scalar partial dif\/ferential equations are formally integrable. But formal integrability need not always be given, as the example of a system of PDEs describing the isometries of a~Riemannian metric shows, which is formally integrable if and only if the corresponding curvature is constant. See Pommaret \cite[Introduction, p.~9]{pommPDEGT} for more information.
\end{enumerate}
\end{Remark}

An important purely analytic consequence of formal integrability is given by the highly nontrivial Theorem~\ref{analytic} below, which essentially states
that if all underlying data are real analytic, then formal integrability implies the existence of local analytic solutions with prescribed f\/inite order
Taylor expansions. Theorem~\ref{analytic} goes back to Goldschmidt~\cite{gold} and heavily relies on (cohomological) results by Spencer~\cite{spencer} and
Ehrenpreis--Guillemin--Sternberg~\cite{ehren}. This result can also be regarded as a variant of Michael Artin's approximation theorem~\cite{artin} or as an interpretation of the Cartan--K\"ahler theorem, cf.~\cite[Theorem~4.4.3]{pomm78} and~\cite{bryant}.

\begin{Theorem}\label{analytic} Assume that $X$ is real analytic, that $\pi$ is a real analytic fiber bundle $($then so is~$\pi_k)$, and that $\IEE\subset \Jet^k(\pi)$ is formally integrable such that in fact $\pi_{k \,|\, \IEE}\colon \IEE\to X$ is a real analytic fibered submanifold of~$\pi_k$. Then, for every $l\in\IN$ and $a\in \IEE^{(l)}$ there exists an open neighborhood $U\subset X$ of $\pi_{k+l}(a)$ and a real analytic solution $\psi\in\Gamma^{\infty}(U;\pi)$ of $\IEE$ such that $\jet^{k+l}_{\pi_{k+l}(a)}\psi=a$.
\end{Theorem}

\begin{proof} This result follows directly from Theorem~9.1 in~\cite{gold} (in combination with Proposition~7.1 therein).
\end{proof}

Now let $\IEE\subset \Jet^k(\pi)$ be a formally integrable partial dif\/ferential equation. Then we can def\/ine a subset $\IEE^{(\infty)}\subset \Jet^{\infty}(\pi)$ by $\IEE^{(\infty)}:= \pi^{-1}_{\infty,k} (\IEE)$. Inductively, one checks that the maps~(\ref{initial}) restrict to surjective maps
\begin{alignat*}{3}
& \IEE^{(\infty)} \longrightarrow \IEE^{(i)},\qquad && a\longmapsto\pi_{k+i,\infty}(a),\qquad i\in\IN, & \\ 
& \IEE^{(\infty)} \longrightarrow X,\qquad && a\longmapsto \pi_{k-1,\infty}(a), & 
\end{alignat*}
so that
\begin{gather*}
\IEE^{(\infty)}= \bigcap_{i\in\IN} \pi^{-1}_{\infty,k+i} (\IEE) .
\end{gather*}

In other words, this means that axioms \ref{axiomsubmfd} to \ref{axiomsubmers} are satisf\/ied for the subset $\IEE^{(\infty)} \subset \Jet^{\infty}(\pi)$ and the smooth projective representation $\big( \Jet^i(\pi),\pi_{i,j}, \pi_{i,\infty} \big)$. Hence, one readily obtains

\begin{Propanddef}\label{defS}
Let $\IEE\subset \Jet^k(\pi)$ be a formally integrable partial differential equation. Then $\big( \Jet^i(\pi),\pi_{i,j}, \pi_{i,\infty} \big)$ induces on $\IEE^{(\infty)}$ the structure of a prof\/inite dimensional submanifold of $(\Jet^{\infty}(\pi),\mathscr{C}^{\infty}_{\pi})$. In view of this fact, $\IEE^{(\infty)}$ will be called the \emph{manifold of formal solutions of $\IEE$.}
\end{Propanddef}

\begin{Remark}
In the above situation, the Cartan distribution on $\Jet^{\infty}(\pi)$ restricts to a well-def\/ined distribution on $\IEE^{(\infty)}$.
\end{Remark}

\subsection{Scalar PDEs and interacting relativistic scalar f\/ields}\label{scal}

Let us f\/irst clarify that throughout Section \ref{scal}, $\pi\colon X\times\IR\to X$ will denote the canonical line bundle.

\subsubsection{A criterion for formal integrability of scalar PDEs}

We now come to the aforementioned result on formal integrability of PDEs. In order to keep the notation simple and in view of the applications that we have in mind, we restrict ourselves in this paper to scalar PDEs.

In the scalar situation, the sheaf of sections of $\pi$ can be identif\/ied with the sheaf of smooth functions on $X$. Recall that the space of smooth functions def\/ined near $p\in X$ is denoted by $\mathscr{C}^{\infty}(p;X)$. Likewise, the space of $k$-th order partial dif\/ferential operators $\mathrm{D}^k(\pi,\pi)$ can be canonically identif\/ied as a linear space with $\mathscr{C}^{\infty}(\Jet^k(\pi))$. Given such an $h\in\mathrm{D}^k(\pi,\pi)$, the space of vector bundle morphisms
\begin{gather*}
\pi^*_{k}\Sym^{k}(\pi_{\IT^*X})\otimes \pi^*_{k,0}\Verti(\pi)\longrightarrow \Verti(\pi_k)
\end{gather*}
over $h$ can be identif\/ied canonically as a linear space (remember here the maps (\ref{iso2}) and (\ref{iso3})) with the space of vector bundle morphisms $\pi^*_{k}\Sym^{k}(\pi_{\IT^*X})\to X\times \IR$ over $\pi_k$. It follows that for every $a\in \Jet^k(\pi)$ the symbol $\sigma(h)$ induces a linear map
\begin{gather*}
 {\sigma(h)}_{|\Sym^{k}\big(\IT^*_{\pi_k(a)}X\big)}\colon \ \Sym^{k}\big(\IT^*_{\pi_k(a)}X\big)\longrightarrow \{\pi_k(a)\}\times \IR .
\end{gather*}

With these preparations, we have:

\begin{Theorem}\label{mai}
 Let $h\in\mathscr{C}^{\infty}(\Jet^k(\pi))$, and recall that
\begin{gather*}
\mathrm{Z}_0(h):= \{a \in \Jet^k(\pi) \, | \, h(a)=0\}.
\end{gather*}
Assume furthermore that the following assumptions are satisfied:
\begin{enumerate}\itemsep=0pt
\item[$(1)$] 
 One has ${\sigma(h)}_{|a}\ne 0$ for all $a\in \Jet^k(\pi)$.

\item[$(2)$] 
 The map $\mathrm{Z}_0(h)^{(1)}\to \mathrm{Z}_0(h)$, $a\mapsto \pi_{k,k+1}(a)$ is surjective.
\end{enumerate}
Then $\mathrm{Z}_0(h)\subset\Jet^k(\pi)$ is a formally integrable partial differential equation on $\pi$.
\end{Theorem}

\begin{Remark}
 Note that assumption (1) together with Proposition \ref{susub} implies that $\mathrm{Z}_0(h)$ indeed is a partial dif\/ferential equation, as $\sigma(h)$ is surjective.
\end{Remark}

\begin{proof}[Proof of Theorem \ref{mai}] The seemingly short proof that we are going to give actually combines two heavy machineries: The already mentioned abstract cohomological criterion for formal integrability of partial dif\/ferential equations from~\cite{gold}, with a highly nontrivial reduction result for the cohomology of Cohen--Macaulay symbolic systems~\cite{krug2}. There seems to be no reasonable elementary proof of Theorem~\ref{mai}.

First observe that (1) is equivalent to
\begin{enumerate}\itemsep=0pt
\item[$(1')$]
 For all $a\in \Jet^k(\pi)$ there exists $v\in \IT^*_{\pi_k(a)}X$ such that $\sigma(h)(v^{\otimes k})\ne 0$.
\end{enumerate}
Namely, if $\sigma_a(h)(v^{\otimes k}) = 0$ for all $v\in \IT^*_{\pi_k(a)}X$, then $\sigma_a(h) =0$ by polarization.

To prove our claim, let an arbitrary $a\in\Jet^k(\pi)$ be given and pick some $v\in\IT^*_{\pi_k(a)} X$ with $\sigma(h)(v^{\otimes k})\ne 0$. Then, in the terminology of~\cite{krug1},
\begin{gather*}
 V^*:= \mathbb{C} v \subset\big(\IT^*_{\pi_k(a)} X\big)_{\mathbb{C}}
\end{gather*}
is a one-dimensional noncharacteristic subspace corresponding to the Cohen--Macaulay symbolic system $\mathrm{g}(h;a)$ given by $\mathrm{Z}_0(h)$ over $a$. Thus we may apply Theorem~A from~\cite{krug2} to deduce that all Spencer cohomology groups $\mathrm{H}^{i,j}(\mathrm{g}(h;a))$ except possibly $\mathrm{H}^{0,0}(\mathrm{g}(h;a))$ and $\mathrm{H}^{1,1}(\mathrm{g}(h;a))$ vanish. But now the result follows from combining~(2), \cite[Theorem~8.1]{gold} and \cite[Proposition~7.1]{gold}, once we have shown that the f\/irst prolongation
\begin{gather}\label{aposs}
 \bigcup_{a\in \mathrm{Z}_0(h)}\mathrm{g}(h;a)^{(1)}\longrightarrow \mathrm{Z}_0(h)
\end{gather}
becomes a vector bundle. To prove the latter claim, note f\/irst that $h$ induces the morphism of vector bundles
\begin{gather*}
\sigma(h)^{(1)}\colon \ \pi_k^*\Sym^{k+1}(\pi_{\IT^*X})\longrightarrow \pi_k^*\IT^*X,
\end{gather*}
which, for every $a\in \Jet^k(\pi)$, is given by
\begin{align*}
 {\sigma(h)^{(1)}}_{|\Sym^{k+1}\big(\IT^*_{\pi_k(a)}X\big)}\colon \
 & \Sym^{k+1}\big(\IT^*_{\pi_k(a)}X\big)\longrightarrow \IT^*_{\pi_k(a)}X,\\
 & v_1\odot\cdots\odot v_{k+1}\longmapsto \underbrace{{\sigma(h)} (v_2\odot\cdots\odot v_{k+1})}_{\in\IR} v_1.
\end{align*}
The assumption (1) immediately implies that the latter linear map is surjective (in fact for every $a\in \Jet^k(\pi)$), in particular, as one has~\cite{gold}
\begin{gather*}
\mathrm{g}(h;a)^{(1)}=\ker \Big( {\sigma(h)^{(1)}}_{|\Sym^{k+1}\big(\IT^*_{\pi_k(a)}X\big)}\Big),
\end{gather*}
it follows that (\ref{aposs}) is a vector bundle.
\end{proof}

The assumption (2) from Theorem~\ref{mai} is a technical regularity assumption (which can become tedious to check in applications), whereas the reader should notice that assumption~(1) therein is essentially trivial and means nothing but that the underlying dif\/ferential operator globally is a~``genuine'' $k$-th order operator.

\subsubsection{Interacting relativistic scalar f\/ields}

As an application of Theorem~\ref{mai}, we will now consider evolution equations that correspond to (possibly nonlinearly!) interacting relativistic scalar f\/ields on semi-Riemannian manifolds. To this end, let~$(X,\mathsf{g})$ be a smooth Lorentzian $m$-manifold. The corresponding d'Alembert operator will be written as
\begin{gather*}
\Box_{\mathsf{g}}\colon \ \mathscr{C}^{\infty}(X)\longrightarrow\mathscr{C}^{\infty}(X).
\end{gather*}
With functions $F_1,F_2\in\mathscr{C}^{\infty}(X)$, $K\in\mathscr{C}^{\infty}(\IR)$, we consider the partial dif\/ferential operator $h_{\mathsf{g},F_1,F_2,K}\in\mathscr{C}^{\infty}(\Jet^2(\pi))$ given for $p\in X$, $\varphi\in \mathscr{C}^{\infty}(p;X)$ by
\begin{gather*}
h_{\mathsf{g},F_1,F_2,K}\big(\jet^2_p\varphi\big) := \Box_{\mathsf{g}}\varphi(p) +F_1(p) \varphi(p)+F_2(p) K(\varphi(p)).
\end{gather*}
What we have in mind here is:

\begin{Example}
Let us assume that $m=4$, that $(X,\mathsf{g})$ has a Lorentz signature, and that
\begin{gather*}
F_1=\alpha_1 \operatorname{scal}_{\mathsf{g}}+\alpha_2^2,\qquad F_2\equiv 1,\qquad K=\alpha_3 \underline{K},
\end{gather*}
where $\alpha_1,\alpha_3\in\IR$, $\alpha_2\geq 0$, $\mathrm{scal}_{\mathsf{g}}\in\mathscr{C}^{\infty}(X)$ denotes the scalar curvature of $\mathsf{g}$ and $\underline{K}\in\mathscr{C}^{\infty}(\IR)$. Then $\mathrm{Z}_0(h_{\mathsf{g},F_1,0,K})\subset \Jet^2(\pi)$ describes the on-shell dynamics of a relativistic (real) scalar f\/ield with mass $\alpha_2$, where $\underline{K}$ is the f\/ield self-interaction with coupling strength $\alpha_3$, and where the number $\alpha_1$ is an additional parameter, which is sometimes set equal to zero. For example, $\underline{K}(z)= z^3$ corresponds to what is called a~\emph{$\varphi^4$-perturbation} in the physics literature (since the corresponding potential in the Lagrange density which has $\mathrm{Z}_0(h_{\mathsf{g},F_1,0,K})$ as its Euler--Lagrange equation is given by $V(\varphi)= \varphi^4$). We refer the reader to~\cite{finster} for the perturbative aspects of this equation in the f\/lat $\varphi^4$ case.
\end{Example}

Returning to the general situation, we can now prove the following result on scalar partial dif\/ferential equations on semi-Riemannian manifolds:

\begin{Proposition}\label{sxcal}
In the above situation, the assumptions {\rm (1)} and {\rm (2)} from Theorem~{\rm \ref{mai}} are satisfied by $h_{\mathsf{g},F_1,F_2,K}$. In particular, $\mathrm{Z}_0(h_{\mathsf{g},F_1,F_2,K})\subset \Jet^2(\pi)$ is formally integrable, and the corresponding space of formal solutions canonically becomes a profinite dimensional manifold via Proposition~{\rm \ref{defS}}.
\end{Proposition}

\begin{proof} In view of Theorem \ref{analytic}, we only have to prove that the assumptions~(1),~(2) from Theorem~\ref{mai} are satisf\/ied. To this end, we set $h:=h_{\mathsf{g},F_1,F_2,K}$ and assume $F_2=0$. Firstly, in view of
\begin{gather*}
 {\sigma(h)}_{|\Sym^{2}\big(\IT^*_{\pi_2(a)}X\big)}(v\odot v)=\mathsf{g}^*_{\pi_2(a)}(v,v)\qquad \text{for all} \quad a\in \Jet^2(\pi), \quad v\in\IT^*_{\pi_2(a)} X,
\end{gather*}
assumption (1) is obviously satisf\/ied and $\mathrm{Z}_0(h)$ indeed is a partial dif\/ferential equation.

It remains to prove that the map $\mathrm{Z}_0(h)^{(1)}\to \mathrm{Z}_0(h)$, $a\mapsto \pi_{2,3}(a)$ is surjective. To see this, assume to be given $b\in \mathrm{Z}_0(h)$ and consider a $\mathsf{g}$-exponential manifold chart $\tilde{x}\colon U\to \IR^m$ of $X$ centered at $\pi_2(b)$. Then one gets the trivial f\/ibered chart
\begin{gather*}
 (x,u):=(\tilde{x},\mathrm{id}_\IR)\colon \ U\times \IR\longrightarrow \IR^m\times \IR
\end{gather*}
of $\pi$, and $b\in \mathrm{Z}_0(h)$ means nothing but $b\in \Jet^2(\pi)$ and
\begin{gather*}
 \sum^m_{i,j=1} \mathsf{g}^{ij}(\pi_2(b)) u_{2,1_{ij}}(b)- \sum^m_{i,j,k=1}\mathsf{g}^{ij}(\pi_2(b))\Gamma^{k}_{ij}(\pi_2(b))u_{2,1_k}(b) \\
\qquad{} +F_1(\pi_2(b))u_{2,(0,\dots,0)}(b)+K(u_{2,(0,\dots,0)}(b)) = 0,
\end{gather*}
where $\mathsf{g}^{ij},\Gamma^{k}_{ij}\in\mathscr{C}^{\infty}(U)$ denote the components of the metric tensor and the Christof\/fel symbols of $g$ with respect to $\tilde{x}$, respectively. Noting that Proposition~\ref{oper} implies $\mathrm{Z}_0(h)^{(1)}=\mathrm{Z}_0(h^{(1)})$, one easily f\/inds that some $a\in \Jet^3(\pi)$ is in $\mathrm{Z}_0(h)^{(1)}$, if and only if
\begin{gather*}
\sum^m_{i,j=1}\mathsf{g}^{ij}(\pi_3(a))u_{3,1_{ij}}(a) -\sum^m_{i,j,k=1}\mathsf{g}^{ij}(\pi_3(a))\Gamma^{k}_{ij}(\pi_3(a))u_{1_k}(a) \\
\qquad{} +F_1(\pi_3(a))u_{3,(0,\dots,0)}(a)+K\big(u_{3,(0,\dots,0)}(a)\big)=0,
\end{gather*}
and, for all $l=1,\dots,m$,
\begin{gather*}
\sum^m_{i,j=1}\big(\partial_l\mathsf{g}^{ij}(\pi_3(a))u_{3,1_{ij}}(a)+\mathsf{g}^{ij}(\pi_3(a))u_{3,1_{ijl}}(a)\big)\\
\qquad{} -\sum^m_{i,j,k=1} \big(\partial_l\mathsf{g}^{ij}(\pi_3(a))\Gamma^k_{ij}(\pi_3(a))u_{3,1_{k}}(a)-\mathsf{g}^{ij}(\pi_3(a))\partial_l\Gamma^k_{ij}(\pi_3(a))u_{3,1_{k}}(a)\big) \\
\qquad{} -\sum^m_{i,j,k=1}\mathsf{g}^{ij}(\pi_3(a))\Gamma^k_{ij}(\pi_3(a))u_{3,1_{lk}}(a)+ \partial_l F_1(\pi_3(a))u_{3,(0,\dots,0)}(a) \\
\qquad{} + F_1(\pi_3(a))u_{3,1_{l}}(a)+K'\big(u_{3,(0,\dots,0)}(a)\big)u_{3,1_l}(a) = 0 .
\end{gather*}
Here, we have used $\partial_l:=\f{\partial}{\partial \tilde{x}^l}$. Let us now assume that the signature of $\mathsf{g}$ is given by $(\varepsilon_1,\dots,\varepsilon_m)=(1,-1,\dots,-1)$. The general case can be treated with the same method. We def\/ine some $a\in \Jet^3(\pi)$ by requiring $\tilde{x}_3(a):=\tilde{x}(\pi_2(b))$, and, for $I\in\IN^m_{0,3}$,
\begin{gather*}
u_{3,I} (a):=
\begin{cases}
 u_{2,I}(b), & \text{if $I\in\IN^m_{0,2}$,}\\
 -\sum\limits^m_{i,j=1}\partial_l\mathsf{g}^{ij}(\pi_2(b))u_{2,1_{ij}}(b) +\sum\limits^m_{i,j,k=1}\partial_l\mathsf{g}^{ij}(\pi_2(b))\Gamma^k_{ij}(\pi_2(b))u_{2,1_{k}}(b) \hspace{-20em} \\
\quad{} +\sum\limits^m_{i,j,k=1}\mathsf{g}^{ij}(\pi_2(b))\partial_l\Gamma^k_{ij}(\pi_2(b))u_{2,1_{k}}(b)  +\sum\limits^m_{i,j,k=1}\mathsf{g}^{ij}(\pi_2(b))\Gamma^k_{ij}(\pi_2(b))u_{2,1_{lk}}(b) \hspace{-20em} \\
 \quad{} -\partial_l F_1(\pi_2(b))u_{2,(0,\dots,0)}(b)-F_1(\pi_2(b))u_{2,1_{l}}(b) \hspace{-20em} \\
\quad{} -K'\big(u_{2,(0,\dots,0)}(b)\big)u_{2,1_l}(b), & \text{if $I=1_{11l}$ for some $l=1,\dots,m$,}\\
 0, & \text{else}.
\end{cases}
\end{gather*}
Now we are almost done: Indeed, our construction of $a$ directly gives $\pi_{2,3}(a)=b$, so $\pi_3(a)=\pi_2(b)$. Since we have
\begin{gather*}
 \mathsf{g}^{ij}(\pi_3(a))=\mathsf{g}^{ij}(\pi_2(b))=
 \begin{cases}
 \varepsilon_j, & \text{if $i=j$},\\
 0, & \text{else},
 \end{cases}
\end{gather*}
it follows immediately that $a\in\mathrm{Z}_0(h)^{(1)}$, and the proof is complete, noting that $F_2$ has not played a role in the above argument.
\end{proof}

\appendix

\section{Two results on completed projective tensor products}

Assume to be given two locally convex topological vector spaces $V$ and~$W$, and consider their algebraic tensor product $V \otimes W$. A topology $\tau$ on
$V \otimes W$ is called \emph{compatible} (in the sense of Grothendieck~\cite{GroPTTEN}) or a \emph{tensor product topology}, if the following axioms hold true:
\begin{enumerate}[itemindent=0mm,leftmargin=4em,labelwidth=0mm,labelsep=2mm,align=right,label={\rm (TPT\arabic*)}]\itemsep=0pt
\item $V \otimes W$ equipped with $\tau$ is a locally convex topological vector space which will be denoted by $V \otimes_\tau W$.
\item The canonical map $V \times W \rightarrow V \otimes_\tau W$ is separately continuous.
\item For every equicontinuous subset $A$ of the topological dual $V'$ and every equicontinuous subset $B$ of the topological dual $W'$, the set $A\otimes B := \{ \lambda \otimes \mu \,|\, \lambda \in A , \, \mu \in B\}$ is an equicontinuous subset of $\big( V \otimes_\tau W \big)'$.
\end{enumerate}
 If $\tau$ is a tensor product topology on $V\otimes W$, we denote by $V\widehat{\otimes}_\tau W$ the completion of $V \otimes_\tau W$.

\begin{Example}\label{ExTensorProds}\quad
\begin{enumerate}\itemsep=0pt
\item[a)] 
The \emph{projective tensor product topology} is the f\/inest locally convex vector space topology on $V\otimes W$ such that the canonical map $V \times W \rightarrow V \otimes W$ is continuous, cf.~\cite{GroPTTEN,TreTVSDK}. The projective tensor product topology is denoted by $\pi$. It is generated by seminorms $p_A \otimes_\pi q_B $, where $p_A$, $A \in \mathcal{A}$ and $q_B$, $B \in \mathcal{B}$ each run through a family of seminorms generating the locally convex topology on $V$ respectively $W$, and $ p_A \otimes_\pi q_B$ is def\/ined by
 \begin{gather*}
 p_A \otimes_\pi q_B (z) := \inf \left\{ \sum_{l=1}^n p_A(v_l) \, q_B (w_l) \,|\, z = \sum_{l=1}^n v_l \otimes w_l \right\} .
 \end{gather*}
The seminorm $p_A \otimes_\pi q_B$ is in particular a \emph{cross seminorm}, i.e., it satisf\/ies the relation
 \begin{gather*}
 p_A \otimes_\pi q_B (v \otimes w) = p_A(v) q_B(w) \qquad \text{for all} \quad v\in V, \quad w\in W.
 \end{gather*}
\item[b)] The \emph{injective tensor product topology} on $V\otimes W$, denoted by $\varepsilon$, is the locally convex topology inherited from the canonical embedding $V\otimes W \hookrightarrow \mathcal{B}_s (V_s' \otimes W_s')$, where $\mathcal{B}_s (V' , W')$ denotes the space of separately continuous bilinear forms on the product~$V' \times W'$ of the weak topological duals~$V'$ and~$W'$ endowed with the topology of uniform convergence on products of equicontinuous subsets of~$V'$ and~$W'$. See~\cite{GroPTTEN} and \cite[Section~43]{TreTVSDK} for details.
\end{enumerate}
\end{Example}

\begin{Remark}\quad
\begin{enumerate}\itemsep=0pt
\item[a)] By def\/inition, the $\varepsilon$-topology on $V\otimes W$ is coarser than the $\pi$-topology. If $V$ (or $W$) is a nuclear locally convex topological vector space, then these two topologies coincide, cf.~\cite{GroPTTEN,TreTVSDK}. Since f\/inite dimensional vector spaces over $\IR$ are nuclear, this entails in particular that for f\/inite dimensional $V$ and $W$ the natural vector space topology on $V \otimes W$ coincides with the (completed) $\pi$- and $\varepsilon$-topology.
\item[b)] The projective tensor product, the injective tensor product, and their completed versions are in fact functors, so it is clear what is meant by $f \otimes_\varepsilon g$, $ f \widehat{\otimes}_\pi g$, and so on, where~$f$ and~$g$ denote continuous linear maps.
\end{enumerate}
\end{Remark}

\begin{Theorem}
Let $( V_i )_{i\in\IN}$ and $( W_i)_{i\in\IN}$ be two families of finite dimensional real vector spaces. Denote by $V$ and $W$ their respective product $($within the category of locally convex topological vector spaces$)$, i.e., let
 \begin{gather*}
 V := \prod\limits_{i\in \IN} V_i \qquad \text{and} \qquad W := \prod\limits_{i\in \IN} W_i .
 \end{gather*}
 Then $V$, $W$, and the completed projective tensor product $V \widehat{\otimes}_\pi W$ are nuclear Fr\'echet spaces. Moreover, one has the canonical isomorphism
 \begin{gather} \label{Eq:RepPiTensorProductProducts}
 V \widehat{\otimes}_\pi W \cong \prod_{(k,l) \in \IN \times \IN} V_k \otimes W_l .
 \end{gather}
\end{Theorem}
\begin{proof}
 Since each of the vector spaces $V_i$ and $W_i$ is a nuclear Fr\'echet space, and countable products of nuclear Fr\'echet are again nuclear Fr\'echet spaces by~\cite{TreTVSDK}, the spaces $V$ and $W$ are nuclear Fr\'echet. Moreover, the same argument shows that $V \widehat{\otimes}_\pi W$ is nuclear Fr\'echet, if equation~\eqref{Eq:RepPiTensorProductProducts} holds true. So let us show equation~\eqref{Eq:RepPiTensorProductProducts}. To this end recall f\/irst
 \cite[Section~3.7]{BouAlgI} that there is a canonical injection
 \begin{align*}
 \iota\colon \ & V \otimes W \longhookrightarrow \prod_{(i,j) \in \IN \times \IN} V_i \otimes W_j , \\
 & ( v_i )_{i\in \IN} \otimes ( w_j )_{j\in \IN} \longmapsto ( v_i \otimes w_j )_{(i,j) \in \IN \times \IN} .
 \end{align*}
 Choose norms $p_i \colon V_i \rightarrow \IR$ and $q_i \colon W_i \rightarrow \IR$. The product topology on $\prod_{(i,j) \in \IN \times \IN} V_i \otimes W_j$ then is def\/ined by the sequence of seminorms
 \begin{gather*}
 r_{k,l} \colon \ \prod_{(i,j) \in \IN \times \IN} V_i \otimes W_j \longrightarrow \IR,\qquad
 ( z_{i,j} )_{(i,j) \in \IN \times \IN} \longmapsto ( p_k \otimes_\pi q_l ) (z_{k,l}) .
 \end{gather*}
 The product topology on $V$ is generated by the seminorms
 $p_k^V \colon V \rightarrow \IR$, $\big( v_i \big)_{i\in \IN}\mapsto p_k (v_k)$,
 the topology on $W$ by the seminorms
 $q_l^W \colon W \rightarrow \IR$, $\big( w_i \big)_{i\in \IN}\mapsto q_l (w_l)$.
 Hence, the $\pi$-topology on $V \otimes W$ is generated by the seminorms
 $p_k^V \otimes_\pi q_l^W $. But since these are cross seminorms, one obtains for
 $(v_i)_{i\in \IN} \in V$ and $(w_i)_{i\in \IN} \in W$ the equality
 \begin{align*}
 p_k^V \otimes_\pi q_l^W \big( (v_i)_{i\in \IN} \otimes (w_i)_{i\in \IN} \big) & =
 p_k^V \big( (v_i)_{i\in \IN} \big) q_l^W \big( (w_i)_{i\in \IN} \big)= p_k (v_k) q_l (w_l) \\
 & = p_k \otimes_\pi q_l (v_k \otimes w_l ) =
 r_{k,l} \big( (v_i \otimes w_j)_{(i,j)\in \IN \times \IN}\big) .
 \end{align*}
 This entails $ p_k^V \otimes_\pi q_l^W = r_{k,l} \circ \iota$, or in other words that the $\pi$-topology on $V \otimes W$ coincides with the pull-back of the product topology on $\prod_{(i,j) \in \IN \times \IN} V_i \otimes W_j$ by the embedding~$\iota$. The claim now follows, if we can yet show that the image of $\iota$ is dense in its range. To prove this let $z = ( z_{i,j})_{(i,j)\in \IN \times \IN} $ be an element of the product $\prod_{(i,j) \in \IN \times \IN} V_i \otimes W_j$. Choose representations
 \begin{gather*}
 z_{i,j} = \sum_{l=1}^{n_{i,j}} v_{i,j,l} \otimes w_{i,j,l}, \qquad \text{where} \quad v_{i,j,l}\in V_i,\quad w_{i,j,l}\in V_j.
 \end{gather*}
Put $v_{i,j,l} = 0$ and $w_{i,j,l} = 0$, if $l > n_{i,j}$. Let $\iota_i^V \colon V_i \rightarrow V$ the embedding of the $i$-th factor in $V$, i.e., the map which associates to $v_i \in V_i$ the family $(v_j)_{j\in \IN}$, where $v_j :=0$, if $j \neq i$. Likewise, denote by $\iota_i^W \colon W_i \hookrightarrow W$ the embedding of the $i$-th factor in $W$. Then def\/ine for $n\in \IN$
 \begin{gather*}
 z_n := \sum_{i,j\leq n} \sum_{l\in \IN} \iota_i^V (v_{i,j,l}) \otimes \iota_j^W (w_{i,j,l}) ,
 \end{gather*}
 and note that by construction the sum on the right side is f\/inite. The sequence $(z_n)_{n\in\IN}$ then is a family in $V\otimes_\pi W$. By construction, it is clear that $ \lim\limits_ {n \rightarrow \infty} \iota( z_n ) = z$. The proof is f\/inished.
\end{proof}

\begin{Theorem}\label{ThmAppProjLimTensProd} Assume that $V$ and $W$ are projective limits of projective systems of finite dimen\-sional real vector spaces $(V_i,\lambda_{ij})$ and $( W_i , \mu_{ij})$, respectively. Denote by
\begin{gather*}
\lambda_i \colon \ V \longrightarrow V_i,\qquad \text{respectively by}\qquad \mu_i \colon \ W \longrightarrow W_i,
\end{gather*}
the corresponding canonical maps. The completed $\pi$-tensor product $V \widehat{\otimes}_\pi W$ together with the family of canonical maps
\begin{gather*}
\lambda_i \widehat{\otimes}_\pi \mu_i \colon \ V \widehat{\otimes}_\pi W \longrightarrow V_i \otimes W_i
\end{gather*}
 then is a projective limit of the projective system $( V_i \otimes W_i , \lambda_{ij} \otimes \mu_{ij})$ within the category of locally convex topological vector spaces. Moreover, both $V$ and $W$ are nuclear, hence $V \widehat{\otimes}_\pi W = V \widehat{\otimes}_\varepsilon W $.
 \end{Theorem}

\begin{proof} First observe that $\big( V_i \otimes W_i , \lambda_{ij} \otimes \mu_{ij}\big)$ is a projective systems of f\/inite dimensional real vector spaces, indeed. Next recall that projective limits of nuclear Fr\'echet spaces are nuclear by~\cite{TreTVSDK}. This proves the second claim. It remains to show the f\/irst one. To this end put $\widetilde{V}_0 := V_0$, $\widetilde{W}_0:= W_0$, and denote for every $i\in\IN^*$ by $\widetilde{V}_i $ be the kernel of the map $\lambda_{i-1i}$ and by $\widetilde{W}_i $ the kernel of $\mu_{i-1i}$. Moreover, choose for every $i\in\IN^*$ a splitting $f_i \colon V_{i-1} \rightarrow V_i$ of $\lambda_{i-1i}$, and a~splitting $g_i \colon W_{i-1} \rightarrow W_i$ of $\mu_{i-1i}$. Put
\begin{gather*}
 \widetilde{V} := \prod\limits_{i\in \IN} \widetilde{V}_i \qquad \text{and} \qquad \widetilde{W}:= \prod\limits_{i\in \IN} \widetilde{W}_i .
 \end{gather*}
Let $\pi_i^{\widetilde{V}} \colon \widetilde{V} \rightarrow \widetilde{V}_i$ be the projection onto the $i$-th factor of $\widetilde{V}$, and $\pi_j^{\widetilde{W}} \colon \widetilde{W} \rightarrow \widetilde{W}_j$ the projection on the $j$-th factor of $\widetilde{W}$.

Now we inductively construct $\widetilde{\lambda}_i \colon \widetilde{V} \rightarrow V_i$ and $\widetilde{\mu}_i \colon \widetilde{W}\rightarrow W_i$. First, put $\widetilde{\lambda}_0 := \pi_0^{\widetilde{V}}$ and $\widetilde{\mu}_0 := \pi_0^{\widetilde{W}}$. Next, assume that we have constructed $\widetilde{\lambda}_0, \ldots,\widetilde{\lambda}_j$ and $\widetilde{\mu}_0,\ldots , \widetilde{\mu}_j$ such that for $i\leq k \leq j$
 \begin{gather} \label{Eq:CompCanMaps}
 \widetilde{\lambda}_i = \lambda_{ik} \circ \widetilde{\lambda}_k \qquad \text{and} \qquad \widetilde{\mu}_i = \mu_{ik} \circ \widetilde{\mu}_k .
 \end{gather}
Then we def\/ine $\widetilde{\lambda}_{j+1}\colon \widetilde{V} \rightarrow V_{j+1}$ and $\widetilde{\mu}_{j+1}\colon \widetilde{W} \rightarrow W_{j+1}$ by
 \begin{gather*}
 \widetilde{\lambda}_{j+1} (v) = \pi_{j+1}^{\widetilde{V}} (v) + f_{j+1} \widetilde{\lambda}_j (v)
 \qquad \text{and} \qquad \widetilde{\mu}_{j+1} (w) = \pi_{j+1}^{\widetilde{W}} (w) + g_{j+1} \widetilde{\lambda}_j (w),
 \end{gather*}
 where $v\in \widetilde{V}$, and $w\in \widetilde{W}$. By assumption on $f_{j+1}$ and $g_{j+1}$ one concludes that
 \begin{gather*}
 \widetilde{\lambda}_j = \lambda_{j+1j} \circ \widetilde{\lambda}_{j+1} \qquad \text{and} \qquad \widetilde{\mu}_j = \mu_{j+1j} \circ \widetilde{\mu}_{j+1} ,
 \end{gather*}
which entails that equation~\eqref{Eq:CompCanMaps} holds true for $i\leq k \leq j+1$. We now claim that $\widetilde{V}$ together with the family $(\widetilde{\lambda}_i)$ is a projective limit of $( V_i , \lambda_{ij})$, and likewise for $\widetilde{W}$. We only need to prove the claim for $\widetilde{V}$. Let $Z$ be a locally convex topological vector space, and $\nu_i \colon Z \rightarrow V_i$ a family of continuous linear maps such that $\nu_i = \lambda_{ij} \circ \nu_j $ for $i\leq j$. Put for every $z\in Z$
 \begin{gather*}
 \widetilde{\nu}_0 (z) := \nu_0(z) \qquad \text{and} \qquad \widetilde{\nu}_i (z) := \nu_i (z) - f_i ( \nu_{i-1} (z))) \qquad \text{for} \quad i\in \IN^*.
 \end{gather*}
 Then $\widetilde{\nu}_i (z)\in \widetilde{V_i}$ for all $i\in \IN$, and
 \begin{gather*}
 \nu \colon Z \ \longrightarrow \widetilde{V}, \qquad z \longmapsto \big( \widetilde{\nu}_i (z) \big)_{i\in \IN}
 \end{gather*}
 is well-def\/ined, linear, and continuous. Moreover, it follows by induction on $i\in\IN$ that
 \begin{gather*}
 \widetilde{\lambda}_i \nu = \nu_i.
 \end{gather*}
 For $i=0$ this is clear, so assume that we have shown this for some $i\in \IN$. Then, for $z\in Z$,
 \begin{gather*}
 \widetilde{\lambda}_{i+1} \nu (z) = \nu_{i+1} (z) - f_{i+1} ( \nu_i (z)) + f_{i+1} \widetilde{\lambda}_i ( \nu (z))= \nu_{i+1} (z),
 \end{gather*}
which f\/inishes the inductive argument. Assume that $\nu' \colon Z \rightarrow \widetilde{V}$ is another continuous linear map such that $\widetilde{\lambda}_i \nu' = \nu_i$ for all $i\in \IN$.
 First, this entails that
 \begin{gather*}
 \pi_0^{\widetilde{V}} \nu' = \widetilde{\lambda}_0 \nu' = \nu_0 = \widetilde{\nu}_0.
 \end{gather*}
 Assume that $ \pi_i^{\widetilde{V}} \nu' = \widetilde{\nu}_i$ for some $i\in \IN$. Then
 \begin{gather*}
 \pi_{i+1}^{\widetilde{V}} \nu' = \widetilde{\lambda}_{i+1} \nu' - f_{i+1}\widetilde{\lambda}_i \nu' = \nu_{i+1} - f_{i+1} \nu_i = \widetilde{\nu}_{i+1} .
 \end{gather*}
 Hence, one obtains, for all $i\in \IN$,
 \begin{gather*}
 \pi_i^{\widetilde{V}} \nu' = \widetilde{\nu}_i = \pi_i^{\widetilde{V}} \nu,
 \end{gather*}
which proves $\nu' = \nu$. So $\widetilde{V}$ is a projective limit of $( V_i , \lambda_{ij})$, and $\widetilde{W}$ a projective limit of $(W_i , \mu_{ij})$. Moreover, $\widetilde{V}$ is canonically isomorphic to~$V$, and~$\widetilde{W}$ to~$W$. The theorem is now proved, if we can show that $\widetilde{V} \otimes_\pi \widetilde{W}$ together with the family of canonical maps $ \widetilde{\lambda}_i \widehat{\otimes}_\pi \widetilde{\mu}_i \colon \widetilde{V} \otimes_\pi \widetilde{W} \rightarrow V_i \otimes W_i$ is a projective limit of the projective system $( V_i \otimes W_i , \lambda_{ij} \otimes \mu_{ij})$. But this is clear, since by the preceding theorem,
 \begin{gather*}
 \widetilde{V} \otimes_\pi \widetilde{W} \cong \prod_{(i,j)\in \IN \times \IN} \widetilde{V}_i \otimes \widetilde{W}_j
 \cong \lim_{\longleftarrow \atop k \in \IN} \prod_{(i,j)\in \IN \times \IN \atop i,j \leq k} \widetilde{V}_i \otimes \widetilde{W}_j
 \end{gather*}
 and, for $k \in \IN$,
 \begin{gather*}
 \prod_{(i,j)\in \IN \times \IN \atop i,j \leq k} \widetilde{V}_i \otimes \widetilde{W}_j \cong
 \prod_{i \leq k} \widetilde{V}_i \otimes \prod_{j \leq k} \widetilde{W}_j \cong V_k \otimes W_k .\tag*{\qed}
 \end{gather*}\renewcommand{\qed}{}
\end{proof}

\subsection*{Acknowledgements}
The f\/irst named author (B.G.) is indebted to W.M.~Seiler for many discussions on jet bundles, and would also like to thank B.~Kruglikov and A.D.~Lewis for helpful discussions. B.G.~has been f\/inancially supported by the SFB 647: Raum--Zeit--Materie, and would like to thank the University of Colorado at Boulder for its hospitality. The second named author (M.P.) has been partially supported by NSF grant DMS 1105670 and by a Simons Foundation collaboration grant, award nr.~359389. M.P.~would also like to thank Humboldt-University, Berlin and the Max-Planck-Institute for Mathematics of the Sciences, Leipzig for their hospitality. Last but not least the authors thank the anonymous referees for constructive advice which helped to improve the paper.

\pdfbookmark[1]{References}{ref}
\LastPageEnding

\end{document}